\documentclass[11pt]{amsart}   	
								
\usepackage{amssymb}

\usepackage{xcolor}
\usepackage[all]{xy}

\usepackage{geometry}
\usepackage[parfill]{parskip}    
\usepackage[all]{xy}
\usepackage{graphicx}
\usepackage{amssymb, amsmath, amsthm}
\usepackage{mathtools}
\usepackage{epstopdf}
\usepackage{url}
\usepackage{soul}
\usepackage{calligra}
\usepackage{fancyvrb}
\usepackage[final]{showkeys}

\usepackage[shortlabels]{enumitem}

\usepackage{tikz}\usetikzlibrary{matrix, cd}
\usepackage{tikz-cd}
\usepackage{xcolor}

\usepackage{tensor}

\usepackage{float}
                                                       
\usepackage[pagebackref]{hyperref}

\usepackage{mathabx,epsfig}

\DeclareMathAlphabet{\mathscr}{T1}{calligra}{m}{n}
\theoremstyle{plain}
\newtheorem{theorem}{Theorem}[subsection]
\newtheorem{proposition}[theorem]{Proposition}
\newtheorem{corollary}[theorem]{Corollary}
\newtheorem{lemma}[theorem]{Lemma}
\theoremstyle{definition}
\newtheorem{definition}[theorem]{Definition}
\theoremstyle{remark}
\newtheorem{remark}[theorem]{Remark}

\newcommand{\End}{\mathcal{E}nd}
\DeclareMathOperator{\Sym}{Sym}

\DeclareMathOperator{\rk}{rk}

\DeclareMathOperator{\Pic}{Pic}

\DeclareMathOperator{\tr}{tr}
\DeclareMathOperator{\Res}{res}
\DeclareMathOperator{\Ker}{Ker}
\DeclareMathOperator{\Ima}{Im}

\newcommand{\id}{\ensuremath{\text{Id}}}

\newcommand{\cC}{\mathcal{C}}
\newcommand{\cB}{\mathcal{B}}
\newcommand{\cE}{\mathcal{E}}
\newcommand{\cO}{\mathcal{O}}
\newcommand{\cM}{\mathcal{M}}
\newcommand{\cD}{\mathcal{D}}

\newcommand{\cL}{\mathcal{L}}

\newcommand{\cA}{\mathcal{A}}

\newcommand{\calC}{\mathcal{C}}

\newcommand{\GG}{\mathbb{G}}
\newcommand{\LL}{\mathcal{L}}
\newcommand{\MM}{\mathcal{M}}

\newcommand{\OO}{\mathcal{O}}

\newcommand{\higgs}{\operatorname{H}}
\newcommand{\semis}{\operatorname{ss}}
\newcommand{\stable}{\operatorname{s}}
\newcommand{\HH}{\mathcal{H}}
\newcommand{\SL}{\operatorname{SL}}
\newcommand{\SU}{\operatorname{SU}}
\newcommand{\GL}{\operatorname{GL}}
\newcommand{\PGL}{\operatorname{PGL}}
\newcommand{\ev}{\operatorname{ev}}

\newcommand{\cX}{\mathcal{X}}
\newcommand{\cY}{\mathcal{Y}}

\newcommand{\ad}{\mathrm{ad}}

\newcommand{\diffn}{\mathcal{D}^{(n)}}
\newcommand{\diffone}{\mathcal{D}^{(1)}}
\newcommand{\difftwo}{\mathcal{D}^{(2)}}
\newcommand{\diffthree}{\mathcal{D}^{(3)}}

\newcommand{\Hit}{\operatorname{Hit}}

\newcommand{\Id}{\operatorname{Id}}

\newcommand{\atalg}{\mathcal{A}}

\VerbatimFootnotes

\newcommand{\doi}[1]{\textsc{doi}: \href{http://dx.doi.org/#1}{\nolinkurl{#1}}}

\title{The Hitchin Connection in Arbitrary Characteristic}
\author{Thomas Baier}
\address{Thomas Baier\\ CAMGSD\\
Instituto Superior T\'ecnico\\
Av. Rovisco Pais\\
1049-001 Lisboa\\
Portugal}
\email{tbaier@math.tecnico.ulisboa.pt}
\author[Michele Bolognesi]{Michele Bolognesi}
\address{Michele Bolognesi\\
Institut Montpelli\'erain Alexander Grothendieck \\ UMR 5149 CNRS \\ Universit\'e de Montpellier\\
Place Eug\`ene Bataillon\\
34095 Montpellier Cedex 5\\ France}
\email{michele.bolognesi@umontpellier.fr}
\author{Johan Martens}
\address{Johan Martens\\ School of Mathematics and Maxwell Institute\\ The University of Edinburgh\\ Peter Guthrie Tait Road\\ Edinburgh EH9 3FD\\ United Kingdom}
\email{johan.martens@ed.ac.uk}
\author{Christian \textsc{Pauly}}
\address{Christian Pauly \\ Laboratoire de Math\'ematiques J.A. Dieudonn\'e \\ UMR  7351 CNRS \\ Universit\'e de Nice Sophia-Antipolis \\ 06108 Nice Cedex 02, France}
\email{pauly@unice.fr}

\thanks{TB was supported in part by FCT/Portugal through the projects UID/MAT/04459/2013 and PTDC/MAT-GEO/3319/2014.  JM was supported in part by EPSRC grant EP/N029828/1. CP was supported in part by
the Marie Curie project GEOMODULI of the programme FP7/PEOPLE/2013/CIG, project number 618471.}

\date{\today}							

\begin{document}

\begin{abstract}
We give an algebro-geometric construction of the Hitchin connection, valid also in positive characteristic (with a few exceptions).  A key ingredient is a substitute for the Narasimhan-Atiyah-Bott K\"ahler form that realizes the Chern class of the determinant-of-cohomology line bundle on the moduli space of bundles on a curve.  As replacement we use an explicit realisation of the Atiyah class of this line bundle, based on the theory of the trace complex due to Beilinson-Schechtman and Bloch-Esnault.
\end{abstract}

\maketitle

\section{Introduction}
\subsection{}The Hitchin connection was originally introduced in \cite{hitchin:1990}, with a two-fold motivation.  The first was an elucidation of the $2+1$ dimensional topological quantum field theory proposed by Witten to explain the polynomial Jones invariants for knots \cite{witten:1989, atiyah:1990}.  The second was the question of the dependency of the geometric quantisation of a symplectic manifold on the choice of polarisation.

In a beautiful construction, Hitchin exhibited a flat projective connection on the bundles of non-abelian theta functions over the base of a family of compact Riemann surfaces.  For a fixed Riemann surface, the corresponding vector space can be understood to be the geometric quantisation of the moduli space of flat unitary connections on the underlying surface.  The latter carries a canonical symplectic structure, but the complex structure on the surface also equips the moduli space with a K\"ahler polarisation, and the connection indicates precisely how the quantisation varies. 

Even though the construction of the connection uses analytic and K\"ahler techniques throughout, it was already observed by Hitchin that the end result could entirely be interpreted in terms of algebraic geometry, and should in fact hold in positive characteristic as well (see \cite[\S 5]{hitchin:1990a}).  
This in itself is not too surprising, bearing in mind that one of the sources of inspiration for Hitchin was the work of Welters {\cite{welters:1983}}, which generalised the heat equation that (abelian) theta-functions had classically been know to satisfy to positive characteristic.  Welters work was probably the first in which a cohomological approach to heat equations was developed; the non-abelian situation is quite a bit more involved, however. 

The aim of this paper now is to give a new, purely algebro-geometric, construction of the Hitchin connection, without using any analytic or K\"ahler techniques.  This construction works as well in positive characteristic (apart from a few exceptions, see below), which as far as we are aware is a first, for either the Hitchin connection itself or any of the equivalent connections (such as the KZB or TUY/WZW connection from conformal field theory -- see however \cite{schechtman.varchenko:2019} for a recent study of the KZ equation in positive characteristic).  We stress that the construction only involves (finite-dimensional) algebraic geometry, and in particular no infinite-dimensional representation theory -- the only prerequisites needed are covered by \cite{ega}.

Key elements in our construction are a framework for connections coming from heat operators in algebraic geometry, due to van Geemen and de Jong \cite{vangeemen.dejong:1998}, as well as a substitute for the Narasimhan-Atiyah-Bott K\"ahler form \cite{narasimhan:1970, atiyah.bott:1983}, which according to Quillen \cite{quillen:1985} realizes the Chern class of the determinant-of-cohomology line bundle.  The serendipitous similarity between this K\"ahler form and the quadratic part of the Hitchin system were crucially used in \cite{hitchin:1990} to obtain the Hitchin connection in the complex case.  

We compensate for the absence of this K\"ahler form by interpreting the cohomology class of the line bundle as an Atiyah class. This difference in guaranteeing the  cohomological conditions of the Theorem of van Geemen and de Jong forms the bulk of our work.
  An essential ingredient of our construction is the description of the Atiyah algebra of the theta line bundle over the moduli space in terms of the first direct image of the Atiyah algebra of a universal bundle (Theorem \ref{maintracecompl}). A complete proof is given in section \ref{sectionbigproof} and in appendices \ref{appendixtracecomplex} and \ref{appendixsplitting}, whose aim is to give a simplified and self-contained presentation of the results used in the proof of this theorem, i.e. the theory of the trace complex \cite{beilinson.schechtman:1988}, \cite{bloch.esnault:2002} and some additional inputs worked out in \cite{sun.tsai:2004}, describing the behaviour of the above objects when replacing a universal bundle by its endomorphism bundle. We observe that the paper \cite{sun.tsai:2004} also describes a construction of the Hitchin connection, but the strategy in \cite{sun.tsai:2004} is different from ours: they construct the Hitchin connection by relying on another argument from \cite{faltings:1997}, whereas our approach seeks to verify directly the van Geemen--de Jong criterion for the liftability of a symbol map to a heat operator.

\subsection{} At this point, we would like to make a few comments on the relationship of this work to the existing literature.  
As already mentioned, we will follow the algebro-geometric framework of van Geemen and de Jong \cite{vangeemen.dejong:1998} for connections induced by heat operators.  This provides a purely cohomological criterion for the existence of a heat operator with a prescribed symbol map.
 
 In \cite[\S 2.3.8]{vangeemen.dejong:1998} van Geemen and de Jong show how their framework of connections induced by heat-operators easily re-captures
Welters' construction of the Mumford-Welters projective connection on bundles of theta functions.  
The main point of their work is to use 
 this framework (which we resume below in Theorem \ref{vgdj}) to construct a Hitchin connection (in complex algebraic geometry) in the particular case of rank $2$ bundles on genus $2$ curves (which was excluded from Hitchin's original work, and indeed from ours as well).  They do not re-establish the Hitchin connection in all other cases though, and in this sense 
 the present paper exactly complements their work.

 We remark that several other algebro-geometric descriptions of connections on bundles of non-abelian theta functions have appeared in the literature -- e.g. \cite{faltings:1997, ramadas:1998, ginzburg:1995, ran:2006,sun.tsai:2004,ben-zvi.frenkel:2004}.  It is not always clear however exactly how these connections are related, see e.g. \cite{faltings-vs-hitchin}, and for various reasons they are all restricted to characteristic zero. None also directly use the framework of van Geemen and de Jong. We remark that many of the properties of Hitchin's original connection like e.g. monodromy \cite{laszlo.pauly.sorger:2013} or projective flatness of strange duality maps \cite{belkale:2009} have been proved with representation-theoretical methods, more precisely by using its equivalence, due to Laszlo \cite{laszlo:1998}, with the TUY/WZW connection on spaces of conformal blocks \cite{TUY:1989, tsuchimoto:1993}.
   For most of the cited works the relationship with conformal blocks is undeveloped (they have of course other motivations: e.g. \cite{sun.tsai:2004}, which together with \cite{ginzburg:1995} is probably closest to our approach, is particularly focused on the logarithmic description of the connection as the curves degenerate to nodal singularities).
We therefore thought it useful to establish the Hitchin connection itself, in the original context (moduli of bundles with trivial determinant over curves), in a purely algebro-geometric way that nevertheless manifestly gives the same connection as Hitchin, and to which Laszlo's theorem immediately applies.  For completeness, we mention that there are several other constructions in the literature of a differential geometric or K\"ahler nature, e.g. \cite{andersen.gammelgaard.lauridsen:2012,axelrod.dellapietra.witten:1991,scheinost.schottenloher:1995}.

We want to mention that (because of Laszlo's theorem) the term \emph{Hitchin connection} is often loosely employed to refer to any of a number of equivalent projective connections.  We shall use it in a much stricter sense however, as a connection arising through a heat operator with a prescribed symbol map (see below).

In this context the terminology \emph{non-abelian theta functions} is frequently used (including by us), even though that is in fact slightly misleading.  Our construction of the connection only works for moduli spaces of bundles with trivial determinant, or equivalently, $\SL(r)$-principal bundles.  At various places the (semi-)simplicity is crucial, and as far as we are aware there is currently no construction that works immediately for arbitrary reductive groups.  Indeed, a connection for moduli of $\GL(r)$-principal bundles was crucially needed in \cite{belkale:2009}, but this was created out of an $\SL(r)$-connection and an (abelian) $\mathbb{G}_m$-connection.

\subsection{} As a motivation for looking at the Hitchin connection 
from a purely algebro-geometric point of view, we would like to highlight three contexts.  The first is the Grothendieck-Katz $p$-curvature conjecture \cite{katz:1972}, which (roughly speaking) claims that every algebraic connection which is formulated in sufficient generality and has vanishing $p$-curvature when reduced mod $p$ for almost all $p$ should have finite monodromy in the complex case.  Presumably motivated by this conjecture it was originally expected (see \cite[\S 7]{brylinski.mclaughlin:1994}) that the Hitchin connection would have finite monodromy.  However, it was shown by Masbaum in \cite{masbaum:1999} that, for rank $2$, the image of the corresponding projective representation of the mapping class group will, for all genera and almost all levels, contain elements of infinite order.  This came somewhat as a surprise, as the connection for abelian theta-functions was well known to have finite monodromy from Mumford's approach through theta groups.   Masbaum was working with a skein-theoretic approach to these representations, but the equivalence of this picture with the Hitchin connection follows from the work of Andersen and Ueno \cite{andersen.ueno:2015} combined with Laszlo's theorem.  Masbaum's result was also directly re-derived in an algebro-geometric context by Laszlo, Sorger, and the fourth named author \cite{laszlo.pauly.sorger:2013}.  We hope that our construction can be a starting point for investigating the $p$-curvature of the Hitchin connection.

The second is the question of integrality of TQFTs, and the related topic of modular representations of the mapping class group.  Various results have been obtained here through a skein-theoretic approach, cfr. \cite{gilmer:2004, gilmer.masbaum:2007, gilmer.masbaum:2014, gilmer.masbaum:2017}, but so far a geometric counterpart is missing.  We again hope that the current work can help shed light on these issues.

Finally we would like to mention various generalisations of the connection constructed here, by looking at variations of the moduli problem of vector bundles on curves.  A minor variation is by looking at moduli spaces of $G$-principal bundles, where $G$ is a semi-simple group.  One could also equip the curve with marked points, and look for parabolic structures of the bundle at these points.  All of these can be understood as special cases of the moduli problem of $\mathcal{G}$-torsors, where $\mathcal{G}$ is a parahoric Bruhat-Tits group scheme over the curve (see e.g. \cite{pappas.rapoport:2010,heinloth:2010,balaji-seshadri:2015}).  We hope to come back to the Hitchin connection in this generality in the near future, and expect that the construction developed in this paper, bypassing the need for an explicit description of a K\"ahler form, will facilitate this.

\subsection{} The rest of the paper is organised as follows. 
In Section \ref{recaphitchin} a summary of Hitchin's work is given, explaining the context of variation of K\"ahler polarisation in geometric quantisation.  There are essentially two parts to this: a general framework that gives conditions under which a projective connection exists (Theorem \ref{mainHitchin}), and a discussion of why these conditions are satisfied in the case of moduli spaces of flat unitary connections on surfaces. Though none of what follows later logically depends on this, we nevertheless wanted to include a brief overview of Hitchin's original construction to highlight the extent to which our exposition parallels his.  

The remainder of the paper is then concerned with our algebro-geometric construction of the Hitchin connection.  In Section \ref{contextvgdj}, after a quick review of Atiyah sequences and Atiyah classes, the notion of heat operators, their relations to connections, and the main framework of van Geemen and de Jong is given (Theorem \ref{vgdj}).  We present the latter as a counterpart to Theorem \ref{mainHitchin}, and for completeness we have included a proof of it and of Hitchin's flatness criterion (Theorem \ref{thm_flatness}), to highlight that these results hold in arbitrary characteristic, as the  original discussion in \cite{vangeemen.dejong:1998} was strictly speaking just in a complex context.

Section \ref{mainconstruction} then goes on to show that the conditions of Theorem \ref{vgdj} are indeed satisfied, culminating in Theorem \ref{existenceconnection}.  The primary tool to this end is Proposition \ref{phi-rho-L}, and most of the rest of the section is essentially a (necessarily lengthy) \emph{mise en place} to obtain this result.  As stated above, the key element is Theorem \ref{maintracecompl}, which realizes the Atiyah class of the determinant-of-cohomology line bundle as a particular extension, given as the first derived functor of the push down of the dual of the Atiyah sequence of the universal bundle on the moduli space of bundles.  This provides an analogue to the theorem of Quillen that realizes the Chern class of the line bundle as a particular K\"ahler form.  Just as in Hitchin's original approach, it is this particular realisation that allows us to verify the cohomological conditions of Theorem \ref{vgdj}.
Theorem \ref{maintracecompl} is itself obtained from a variation on the theory of the trace complex, of which we give a self-contained account in Appendix \ref{appendixtracecomplex}.  The proof of  Theorem \ref{maintracecompl} takes up Section \ref{sectionbigproof}.   Finally, the other appendices contain proofs of various facts we use in the main body of the article, but for which we could not find references in the generality we needed.
\subsection{} To finish the introduction, we state the necessary restrictions on the characteristic $p$ of the base field $\Bbbk$, and their sources. The first limitation that we encounter is due to the use of the trace and the trace pairing:
\[\begin{tikzcd}
\operatorname{tr}:\cE nd(E) \ar[r] & \cO, &\operatorname{Tr}: \cE nd(E) \times \cE nd(E) \ar[r] & \cO.
\end{tikzcd}
\]
We need these to behave similarly as they do in characteristic zero.  In particular we want the trace $\operatorname{tr}$ to split equivariantly , i.e. $\cE nd(E)=\cE nd^0(E)\oplus \cO$, where $\cE nd^0(E)$ is the kernel of $\operatorname{tr}$.  This is induced from an $\SL(r)$-equivariant splitting of the short exact sequence of Lie algebras 
\[\begin{tikzcd} 0\ar[r] & \mathfrak{sl}(r)\ar[r] & \mathfrak{gl}(r)\ar[r] & \Bbbk \ar[r] & 0,\end{tikzcd}\] which requires $p \nmid r$.  Secondly, we want the trace pairing $\operatorname{Tr}$, which is non-degenerate for all possible characteristics $p$ and $r=rk(E)$, to remain non-degenerate when restricted to $\cE nd^0(E)\times \cE nd^0(E)$. This is again true if and only if $p \nmid r$.

The second limitation is due to the use of  differential operators (cf. \cite[IV, \S 16.8]{ega}) and their symbols: in characteristic $p>0$ one considers the algebra of differential operators associated to the Atiyah algebra $\mathcal{D}^{(1)}_{\mathcal{M}/S}(L)$ and defined as a quotient of its universal enveloping algebra -- see \cite[1.1.3]{beilinson.schechtman:1988}. Up to order $k=p-1$ these however coincide with $\mathcal{D}^{(k)}_{\mathcal{M}/S}(L)$, and we have the symbol map to $\Sym^k T_{\mathcal{M}/S}$ with its usual properties at our disposal. As the construction of connections via heat operators uses second order operators and their symbols, we exclude characteristic 2; in the flatness criterion also third-order symbols appear, hence there we also exclude $p=3$.

Furthermore, we also use trace complexes; the original reference avoids positive characteristic, but as we use only part of the theory we check in Appendix \ref{appendixtracecomplex} that everything works with the restrictions already in place: in order for the  residue $\widetilde{\Res}$ from \cite[page 658]{beilinson.schechtman:1988} to be well defined, we need to avoid characteristic 2.

The third and last limitation is due to the formula in Thm. \ref{existenceconnection}, where there is a factor $\frac{1}{r+k}$. Hence we also need to assume that $p\nmid (r+k)$.
\subsection{Acknowledgments} The authors would like to thank J\o rgen Andersen, Prakash Belkale, C\'edric Bonnaf\'e, Najmuddin Fakhruddin, Emilio Franco, Bert van Geemen, Jochen Heinloth, Nigel Hitchin, Gregor Masbaum, Swarnava Muk\-ho\-padh\-yay, Jon Pridham, Brent Pym,  Pavel Safronov, Richard Wentworth and Hacen Zelaci for useful conversations and remarks at various stages of this work.  This work grew out of another project of the first and third named authors that was joint with J\o rgen Andersen, Peter Gothen and Shehryar Sikander -- they thank all three of them for related discussions.

\section{Heat operators and connections - summary of the work of Hitchin}\label{recaphitchin}
We outline in this section the original work of Hitchin that establishes the flat projective connection on bundles of non-abelian theta functions. 
Hitchin's motivation came from geometric quantisation and K\"ahler geometry, and he mainly used analytic or K\"ahler techniques. 
\subsection{Change of K\"ahler polarisation} Inspired by earlier work of Welters \cite{welters:1983}, the Hitchin connection was introduced in \cite{hitchin:1990} in the context of geometric quantisation: given a compact (real) symplectic manifold $(\cM,\omega)$, with pre-quantum line bundle $L$, Hitchin studied how the geometric quantisations with respect to different K\"ahler polarisations were related.  In particular, he gave the following general criterium for the existence of a projective connection on the bundle of quantisations:
\begin{theorem}[Hitchin, {\cite[Theorem 1.20]{hitchin:1990}}] \label{mainHitchin}
Given a family of K\"ahler polarisations on $\cM$, such that for each polarisation we have 
\begin{enumerate}[(a)]
\item The map $$\begin{tikzcd}\cup [\omega]: H^0(\cM,T_{\cM}) \ar[r] &H^1(\cM, \mathcal{O}_{\cM})\end{tikzcd}$$ is an isomorphism (this means that there are no holomorphic vector fields which fix $L$, i.e. $H^0(\cM, \diffone_{\cM}(L))=H^0(\cM, \mathcal{O}_{\cM})$);
\item\label{hitchintwo} For each $s\in H^0(\cM,L)$ and tangent vector $\overset{.}{I}$ to the base of the family there exists a smoothly varying $$A(\overset{.}{I}, s)\in \mathbb{H}^1(\cM, \diffone_{\cM}(L)\overset{.s}{\rightarrow} L)$$
such that the symbol $-i\sigma_1 (A(\overset{.}{I}, s))$ equals the Kodaira-Spencer class $[\overset{.}{I}]$ in\linebreak $H^1(\cM, T_{\cM})$.
\end{enumerate}
Then this defines a projective connection on the bundle of projective spaces $\mathbb{P}(H^0(\cM, L))$ over the base of the family.
\end{theorem}
Here $\diffone_{\cM}(L)$ denotes the sheaf of first order differential operators on $L$ and 
$\sigma_1$ its symbol map to $T_{\cM}$. The map $.s:\diffone_{\cM}(L)\rightarrow L$ is just given by evaluating the differential operators on the section $s$, and $\mathbb{H}^1$ stands for the first hypercohomology group of the two-term complex.

Note that the space of infinitesimal deformations of the pair $(\cM, L)$ is given by\linebreak $H^1(\cM, \diffone_{\cM}(L))$, and likewise the space of infinitesimal deformations of the triple\linebreak $(\cM,L,s)$, for $s\in H^0(\cM,L)$, is given by 
$\mathbb{H}^1(\cM, \diffone_{\cM}(L)\overset{.s}{\rightarrow} L)$ (cfr. \cite[Proposition 1.2]{welters:1983}).
\subsection{Moduli spaces of flat unitary connections} Moreover, Hitchin then showed that the conditions of Theorem \ref{mainHitchin} are satisfied in the case where $(\cM,\omega)$ is the space of flat, unitary, tracefree connections on the trivial rank $r$ bundle over a closed oriented surface $\calC$ of genus $g\geq 2$ (with the exception of the case $r=2, g=2$), and $L=\LL^k$ is a power of the positive generator $\LL$ of its Picard group. This space is not quite a manifold, but its smooth locus is canonically a symplectic manifold, with $\omega$ the Goldman-Karshon symplectic 
form (which uses a Killing form on the Lie algebra of $\SU(r)$).
  
If $\calC$ is equipped with the structure of a Riemann surface (or, equivalently, regarded as a smooth complex projective curve), then $\cM$ can be understood as the moduli space of semi-stable rank $r$ vector bundles with trivial determinant, which is a projective variety.  The symplectic form $\omega$ is then moreover a K\"ahler form, as discussed by Narasimhan \cite{narasimhan:1970} and Atiyah-Bott \cite{atiyah.bott:1983}. By Quillen's theorem \cite{quillen:1985}, the inverse $\LL$ of the determinant-of-cohomology line bundle provides a pre-quantum line bundle.

In particular, we can understand the $A(\overset{.}{I}, s)$ as follows in this situation:  we have the short exact sequence of complexes
\begin{equation}\label{sesofcomplex}\begin{tikzcd}[row sep=small]
0\ar[r] &\diffone_{\cM} (\LL^k) \ar[r] \ar[d, ".s"] & \difftwo_{\cM} (\LL^k)\ar[r] \ar[d, ".s"] & \Sym^2 T_{\cM}\ar[r] \ar[d]& 0\\
0\ar[r] & \LL^k \ar[r] &\LL^k \ar[r] & 0 \ar[r] &0.
\end{tikzcd}\end{equation}
This gives a connecting homomorphism \begin{equation}\label{boundary}\begin{tikzcd}\delta: H^0(\cM, \Sym^2(T_{\cM})) \ar[r] & \mathbb{H}^1(\cM, \diffone_{\cM}(\LL^k)\overset{.s}{\rightarrow} \LL^k).\end{tikzcd}
\end{equation}
On the other hand, the quadratic part of the Hitchin system (which also uses the Killing form) gives, for every holomorphic vector bundle $E$ on $\calC$ with trivial determinant, a map $$\begin{tikzcd}\Sym^2 H^0(\calC,\End^0(E)\otimes K_{\calC}) \ar[r] & H^0(\calC, K^2_{\calC}),\end{tikzcd}$$ where $K_{\calC}$ is the canonical bundle of $\calC$.
Dualizing this, and using Serre duality on $\calC$ gives, for each $E$, a map 
$$\begin{tikzcd}H^1(\calC, T_{\calC}) \ar[r] & \Sym^2  H^1(\calC, \End^0(E)),\end{tikzcd}$$ 
where $\End^0(E)$ is the sheaf of trace-free endomorphisms of $E$.
Since for each stable $E$ the space $H^1(\calC, \End^0(E))$ is the tangent space to the moduli space (in casu $\cM$), we can write this as a map \begin{equation}\begin{tikzcd}\label{rho}\rho:H^1(\calC, T_{\calC}) \ar[r] &H^0(\cM, \Sym^2 T_{\cM}).\end{tikzcd}\end{equation}
Composing this with (\ref{boundary}) gives a linear map $$\begin{tikzcd}A(., s): H^1(\calC, T_{\calC}) \ar[r] &\mathbb{H}^1(\cM, \diffone_{\cM}(\LL^k)\overset{s}{\rightarrow} \LL^k)\end{tikzcd}$$ which depends smoothly on $s$, and which Hitchin shows (after a rescaling by $\frac{1}{r+k}$) to satisfy 
the condition in \ref{hitchintwo} of Theorem \ref{mainHitchin}.
\begin{remark} Some key steps in Hitchin's approach were fundamentally differential geometric or K\"ahler in nature. In particular, the explicit description of the Narasimhan-Atiyah-Bott K\"ahler form, and its similarity to the symmetric two-tensors given by the symbol was crucially used.
\end{remark}

\section{Hitchin-type Connections in Algebraic Geometry}\label{contextvgdj} 
An algebro-geometric framework for connections determined by a heat equation (like the Hitchin connection) was developed by van Geemen and de Jong in \cite{vangeemen.dejong:1998}.   Besides being set in algebraic geometry as opposed to K\"ahler geometry, this description is also more local, in contrast with the infinitesimal framework of Theorem \ref{mainHitchin} of Hitchin (the latter is not a substantional difference however, cfr. \cite[\S 2.3.4]{vangeemen.dejong:1998}). We summarise the main parts and some related prerequisites below. 

From now on, everything will be defined over an algebraically closed field $\Bbbk$ of characteristic different from $2$.  We have to exclude characteristic $2$ for a variety of reasons, but in particular will also split the projection $T_{\cM}^{\otimes 2}\rightarrow \Sym^2 T_{\cM}$ throughout.  In this general section, $\cM \rightarrow S$ will be a smooth morphism of smooth schemes.
\subsection{Atiyah Algebroids, (projective) connections, and Atiyah classes}\label{seqandcon}
Our approach to connections essentially follows Atiyah's seminal exposition \cite{atiyah:1957}, but in this context we will phrase everything in terms of vector bundles rather than work with principal bundles.  
\subsubsection*{Atiyah algebroids}Let $\diffn_{\cM}(E)$ be the sheaf of differential operators of order at most $n$ on a vector bundle $E$ over $\cM$.  The associated symbol map will be denoted $$\sigma_n:\diffn_{\cM}(E)\rightarrow \Sym^n T_{\cM}\otimes \End(E).$$
\begin{definition}
The Atiyah sequence associated to a vector bundle $E\rightarrow \cM$ is the top row of the following diagram
\begin{equation*}
\begin{tikzcd}[row sep=small] 0\ar[r]& \End(E)\ar[r]\ar[d,equal]  & \atalg(E)\ar[r] \ar[d,hookrightarrow]& T_{\cM}\ar[r] \ar[d, hookrightarrow, "-\otimes \Id_E"]&0\\ 0 \ar[r] & \End(E)\ar[r] & \diffone_{\cM}(E)
  \ar[r, "\sigma_1"] & T_{\cM}\otimes \End(E)\ar[r] & 0.\end{tikzcd}
\end{equation*}
The middle term $\atalg(E)$ is called the Atiyah algebroid associated to $E$ (or, strictly speaking, to the frame bundle associated to $E$, which is a $\GL$-principal bundle).
\end{definition}
\begin{definition}
We will denote by $\atalg_{\cM/S}(E)$, the relative Atiyah algebroid associated to a vector bundle $E\rightarrow \cM$, where $\cM$ comes with a morphism $\pi: \cM\rightarrow S$ onto a base scheme $S$.  The associated relative Atiyah sequence is the top row of the following pull-back diagram:
\begin{equation}\label{relatiyahsequence}
  \begin{tikzcd}[row sep=small] 0\ar[r]& \End(E)\ar[r]\ar[d,equal]  & \atalg_{\cM/S}(E)\ar[r] \ar[d,hookrightarrow]& T_{\cM/S}\ar[r] \ar[d, hookrightarrow]&0\\ 0 \ar[r] & \End(E)\ar[r] & \atalg(E)\ar[r] & T_{\cM}\ar[r] & 0.\end{tikzcd}
\end{equation}
where $T_{\cM/S}$ is the subsheaf of vector fields tangent along the
fibers, i.e.,
$$T_{\cM/S} = \Ker(T_{\cM} \to \pi^* T_S).$$
\end{definition}
Finally, we need to define the trace-free Atiyah algebroid for vector bundles with trivial determinant.  Pushing out the standard Atiyah sequence by the trace map $\End(E)\rightarrow \OO$ gives a morphism of the Atiyah sequene of $E$ to that of $\det(E)$.  If the latter is trivial, its Atiyah sequence splits canonically, giving rise a morphism $\tr:\atalg(E)\rightarrow \OO$.  We define the trace free Atiyah algebroid $\atalg^0(E)$ to be the kernel of this map.  This all fits together in a commutative diagram (with exact horizontal rows and left vertical row) 
\begin{equation*}
  \begin{tikzcd}[row sep=small] 0\ar[r]& \End^0(E)\ar[r]\ar[d,hookrightarrow]  & \atalg^0(E)\ar[r, "\sigma_1"] \ar[d,hookrightarrow]& T_{\cM}\ar[r] \ar[d, equal]&0\\ 0 \ar[r] & \End(E)\ar[r] \ar[d, "\tr"] & \atalg(E)\ar[r,"\sigma_1"] \ar[d, "\tr + \sigma_1"]& T_{\cM}\ar[r]\ar[d, equal] & 0 \\ 0\ar[r] & \OO\ar[r] & \atalg(\det(E))\cong \OO\oplus T_{\cM}\ar[r] & T_{\cM}\ar[r] & 0. \end{tikzcd}
\end{equation*} The algebroid $\atalg^0(E)$ can be understood, in the language of principal bundles, as arising from the $\SL(r)$-principal frame bundle of $E$. Analogously there is also a relative version $\atalg^0_{\cM/S}(E)$.

Assuming $p \nmid r$, we have a direct sum decomposition 
$\cE nd(E) = \cE nd^0(E) \oplus \cO_{\cM}$ and we denote by
$q: \cE nd(E) \to \cE nd^0(E)$ the projection onto the first 
direct summand.  In this case, the trace-free Atiyah algebroid is also canonically isomorphic to the \emph{projective} Atiyah algebroid, i.e. the push-out of the standard Atiyah sequence 
by the map $q$ as follows
\begin{equation*}
\begin{tikzcd}[row sep=small] 
0 \ar[r] & \End(E) \ar[r]\ar[d,"q"] & \cA(E) \ar[r]\ar[d] & T_\cM \ar[r]\ar[d,equal] & 0 \\
0 \ar[r] & \End^0(E) \ar[r] & \cA^0(E) \ar[r] & T_\cM \ar[r] & 0.
\end{tikzcd}
\end{equation*} 
We will make this identification throughout.
\subsubsection*{Atiyah classes}
We will also need a relative version of the Atiyah class for a line bundle $L$.  There are a number of ways this can be defined; perhaps the easiest is by taking the top sequence of (\ref{relatiyahsequence}), tensoring it with $\Omega^1_{\cM/S}$, and applying $\pi_*$ to obtain a long exact sequence (of course for line bundles we have canonically $\End(L)\cong\cO$).
\begin{definition}
The image of the identity $\pi_\ast\Id\in\pi_*\big( \Omega^1_{\cM/S}\otimes T_{\cM/S}\big)$ under the connecting homomorphism yields a global section of $R^1\pi_* \big(\Omega^1_{\cM/S} \otimes \End(E)\big)$, which we shall refer to as the \emph{relative Atiyah class}, and denote by $[L]$.
\end{definition}

Note that the connecting homomorphism in the long exact sequence obtained by applying $\pi_*$ to the top sequence of (\ref{relatiyahsequence}) is given by cupping with $[L]$ and contracting.   In the absolute case, the Atiyah class is the obstruction to the existence of a connection on $L$; a similar interpretation holds in the relative case, though we will not use this.  If $\cM$ is complex K\"ahler, $[L]$ is just the relative Chern class.

The following lemma probably dates back to \cite{atiyah:1957}, see e.g. \cite[p. 431]{looijenga:2013}. 
\begin{lemma}\label{extens}
Let $X$ be a smooth algebraic variety, $L$ a line bundle, $k$ a positive integer, then we have an isomorphism of short exact sequences
\begin{equation*}
\begin{tikzcd}[row sep=small]
  0 \ar[r] & \cO_X \ar[r] \ar[d] & \cA(L^{\otimes k}) \ar[r]\ar[d] & T_X \ar[r]\ar[d] & 0 \\
  0 \ar[r] & \cO_X \ar[r, "\frac{1}{k}"] & \cA(L) \ar[r] & T_X \ar[r] & 0. 
\end{tikzcd}
\end{equation*}
\end{lemma}
\subsubsection*{Projective connections}
\begin{definition}
Given a vector bundle $E$ on a variety $\cM$, a \emph{(Koszul) connection} $\nabla$ on $E$ is a $\cO_{\cM}$-linear splitting of the Atiyah algebroid: $$\begin{tikzcd} 0\ar[r]& \End(E)\ar[r] & \atalg(E)\ar[r] &T_{\cM} \ar[r] \ar[l, bend left=30, " \nabla", dashed] & 0.\end{tikzcd}$$  
The connection is said to be \emph{flat} (or integrable) if $\nabla$ preserves the Lie brackets (where the Lie bracket on $\atalg(E)$ is just the commutator of differential operators).
\end{definition}
The Hitchin connection is a projective connection.  There are a number of ways one can encode what a projective connection is: one could think in terms of $\PGL$ principal bundles, or work with the projectivisation $\mathbb{P}(E)$ of $E$, or work with twisted $\cD$-modules (cfr. \cite{beilinson.kazhdan:1990}, \cite[\S 1]{looijenga:2013}).  In our context, the most useful one is the following.
\begin{definition} Given a vector bundle $E$ on $\cM$ as before, a projective connection is a splitting 
$$\begin{tikzcd} 0\ar[r]& \End(E)/\cO_{\cM} \ar[r] & \atalg(E)/{\cO_{\cM}}\ar[r] &T_S\ar[r] \ar[l, bend left=30, dashed, " \nabla "] & 0.\end{tikzcd}$$
It is again flat if $\nabla$ preserves the Lie brackets.
\end{definition}
\subsection{Heat operators} 
Consider a smooth surjective morphism of smooth schemes $\pi:\cM\rightarrow S$, and a line bundle $L\rightarrow \cM$ such that $\pi_\ast L$ is 
locally free, hence a vector bundle. The connection we construct will live on the projectivisation $\mathbb{P}\pi_\ast L$, but everything below will be expressed in terms of vector bundles, not projective bundles. 

We will denote by $\diffn_{\cM/S}(L)$ the subsheaf of $\diffn_{\cM}(L)$ consisting of differential operators of order at most $n$ that are $\pi^{-1}(\mathcal{O}_S)$ linear.  The symbol maps
$$\begin{tikzcd}\sigma_n: \mathcal{D}^{(n)}_{\cM/S}(L) \ar[r] & \Sym^n T_{\cM/S}\end{tikzcd} $$
take values in $\Sym^n T_{\cM/S}$.

We are now interested in the sheaf $$\mathcal{W}_{\cM/S}(L)=\diffone_{\cM}(L)+\difftwo_{\cM/S}(L)\subset \difftwo_{\cM}(L).$$
Besides the second order symbol map $$\begin{tikzcd} \sigma_2: \mathcal{W}_{\cM/S}(L) \ar[r] &\Sym^2T_{\cM/S},\end{tikzcd}$$ on this sheaf of differential operators, there is a subprincipal symbol
\begin{equation}\label{subprincipal} \begin{tikzcd}
  \sigma_S: \mathcal{W}_{\cM/S}(L)  \ar[r] &\pi^\ast T_{S} ,
  \qquad 
  \langle \sigma_S(D),d (\pi^\ast f) \rangle s =  D(\pi^\ast f s) - \pi^\ast f D(s).\end{tikzcd}
\end{equation}
where $s$ is a local section of $L$ and $f$ a local section of $\mathcal{O}_S$;
both well-definedness and the Leibniz rule follow from the property of the second order symbol
\[
 D(fg s) = \langle \sigma_2(D) , df \otimes dg \rangle s + f D(gs)+g D(fs)-fg D(s) .
\]
Thus we have a short exact sequence
\begin{equation}\label{ses_W}
  \begin{tikzcd}0\ar[r] & \diffone_{\cM/S}(L)\ar[r] & \mathcal{W}_{\cM/S}(L) \ar[r,"^{\sigma_S\oplus\sigma_2}"] &\pi^* (T_S)\oplus \Sym^2 T_{\cM/S}\ar[r] & 0.\end{tikzcd}
\end{equation}
We can now define 
\begin{definition}[{\cite[2.3.2]{vangeemen.dejong:1998}}]
A \emph{heat operator} $D$ on $L$ is a $\mathcal{O}_S$-linear map of coherent sheaves $$\begin{tikzcd}D:T_S  \ar[r] &\pi_\ast \mathcal{W}_{\cM/S}(L)\end{tikzcd}$$ such that $\sigma_S \circ 
\widetilde{D}=\id$, where $\widetilde{D}$ is the equivalent (by
adjunction) $\cO_{\cM}$-linear map
$$\widetilde{D} : \pi^* T_S \to \mathcal{W}_{\cM/S}(L).$$
Similarly a \emph{projective heat operator} is a map 
$$\begin{tikzcd}D: T_S  \ar[r] & \left( \pi_\ast \mathcal{W}_{\cM/S}(L) \right) / \mathcal{O}_S.\end{tikzcd}$$
\end{definition}
Given such a heat operator, we refer to $$\begin{tikzcd}
\pi_*(\sigma_2) \circ D: T_S\ \ar[r] &\pi_\ast\Sym^2 T_{\cM/S}\end{tikzcd} $$ as the \emph{symbol} of the heat operator.
  Also a projective heat operator has a well-defined symbol.
\subsection{Heat operators and connections} Any heat operator gives rise to a connection on the locally free sheaf $\pi_* L$, as follows (cfr. \cite[\S 2.3.3]{vangeemen.dejong:1998}).  Given an open subvariety $U\subset S$, and $\theta \in T(U)$, we want a first order differential operator $$\begin{tikzcd}\nabla_{\theta}: \pi_*  L\ \ar[r] &\pi_* L.\end{tikzcd}$$
If $s\in \pi_* L(U)$, we denote by $s$ and $\pi^{-1}(\theta)$ the corresponding sections of $L(\pi^{-1}(U))$ and $\pi^{-1}(T_S)(\pi^{-1}(U))$ respectively.  We can now put
\begin{equation}\label{conn-heat-op}
 \nabla_{\theta}s=D(\pi^{-1}(\theta))(s),
\end{equation}
since the latter indeed corresponds to a section of $\pi_* L(U)$.  Moreover, the Leibniz rule is satisfied since the subprincipal symbol of $D(\pi^{-1}\theta)$ is $\pi^{-1}\theta$, so that for any $f\in \mathcal{O}_S(U)$ we have
\[
\nabla_{\theta}(fs)=D(\pi^{-1}(\theta))(\pi^\ast(f) s) = \pi^\ast(\theta(f)) s+ \pi^\ast(f) D(\pi^{-1}(\theta))(s)= \theta(f) s+ f\nabla_{\theta}s,
\]
so $\nabla_{\theta}$ is indeed a first order differential operator with symbol $\theta$, and hence $\nabla$ is indeed a Koszul connection.  

The connection $\nabla$ will be flat if $D$ preserves the Lie brackets.  If we have a projective heat operator, we still get a projective connection, with the same comment for flatness.

\subsection{A heat operator for a candidate symbol}
As an algebro-geometric counter-part to Hitchin's Theorem \ref{mainHitchin}, van Geemen and de Jong investigated under what conditions a candidate symbol map $$\begin{tikzcd}\rho: T_S \ar[r]& \pi_\ast \Sym^2 T_{\cM/S}\end{tikzcd}$$ actually arises as a symbol of a heat operator, i.e. whether it was possible to find a (projective) heat operator $D$ such that $\rho= \pi_*(\sigma_2) \circ D$.  Before we can state their result we need to recall two maps. The canonical short exact sequence $$\begin{tikzcd}0\ar[r]& T_{\cM/S}\ar[r]& T_{\cM}\ar[r]& \pi^*T_S\ar[r]& 0\end{tikzcd}$$ gives rise to the \emph{Kodaira-Spencer map}
\begin{equation}\label{ks}\begin{tikzcd}\kappa_{\cM/S}: T_{S} \ar[r] & R^1\pi_* T_{\cM/S}.\end{tikzcd}\end{equation}
Similarly, the short exact sequence \begin{equation}\label{sesnoext}\begin{tikzcd}0\ar[r]& T_{\cM/S}\ar[r]& \difftwo_{\cM/S}(L)/\mathcal{O}_{\cM}\ar[r]& \Sym^2 T_{\cM/S}\ar[r]& 0 \end{tikzcd}\end{equation} gives rise to the connecting homomorphism \begin{equation}\begin{tikzcd}\label{mu-L}\mu_{L}:\pi_* \Sym^2 T_{\cM/S}  \ar[r] & R^1\pi_* T_{\cM/S}.\end{tikzcd}
\end{equation}
We can now state 
\begin{theorem}[{van Geemen -- de Jong,\cite[\S 2.3.7]{vangeemen.dejong:1998}}]\label{vgdj} With $L$ and $\pi:\cM\rightarrow S$ as before, we have that
if, for a given $\rho: T_S \rightarrow \pi_\ast \Sym^2 T_{\cM/S}$,  \begin{enumerate}[(a)]
\item \label{vgdj-one} $\kappa_{\cM/S}+\mu_{L} \circ \rho=0,$ 
\item \label{vgdj-two} cupping with the relative Atiyah class \[\begin{tikzcd}\cup [L]: \pi_*T_{\cM/S}\ar[r] &R^1\pi_*\mathcal{O}_{\cM}\end{tikzcd}\] is an isomorphism, and
\item \label{vgdj-three} $\pi_*\mathcal{O}_{\cM}=\mathcal{O}_S$, 
\end{enumerate} then there exists a unique projective heat operator $D$ whose symbol is $\rho$.
\end{theorem}
Note that even though the context of this theorem is entirely algebro-geometric and makes no reference to a symplectic form, the conditions are closely matched with those in Hitchin's Theorem \ref{mainHitchin}: the requirement of cupping with the Chern class being an isomorphism is identical in both cases, whereas from a quadratic symbol $\rho$ satisfying condition \ref{vgdj-one} we recover an element of the hypercohomology group in \ref{mainHitchin}.\ref{hitchintwo} via the long-exact sequence of hypercohomology obtained from (\ref{sesofcomplex}).  Finally, \ref{vgdj-three} is an appropriate weakening of the premise that $\mathcal{M}$ is compact (and connected) in Theorem \ref{mainHitchin}.
\begin{proof}
  Consider the long-exact sequence associated to the short exact sequence (\ref{ses_W}),
  \[
  \begin{tikzcd}[row sep=small]
    0\ar[r] & \pi_\ast \diffone_{\cM/S}(L)\ar[r] & \pi_\ast \mathcal{W}_{\cM/S}(L) \ar[r,"^{\pi_* \sigma_S\oplus \pi_* \sigma_2}"] & T_S\oplus \pi_\ast \Sym^2 T_{\cM/S} \ar[dll,swap,"\delta"] \\
    & R^1 \pi_\ast \diffone_{\cM/S}(L)\ar[r] & R^1 \pi_\ast \mathcal{W}_{\cM/S}(L) \ar[r] & \dots \end{tikzcd}
  \]
  As $\cup[L]$ is the connecting homomorphism in the long-exact sequence associated with the first order symbol map on $\diffone_{\cM/S}(L)$, condition \ref{vgdj-two} guarantees that $\mathcal{O}_S = \pi_\ast \mathcal{O}_{\cM} = \pi_\ast \diffone_{\cM/S}(L)$, i.e. all global first order operators on $L$ along the fibers of $\pi$ are of order zero. Using condition \ref{vgdj-three}, we obtain a commutative diagram with exact rows and columns 
  \[
  \begin{tikzcd}[row sep=small, column sep=tiny]
     & 0 \ar[d] & 0 \ar[d] \\
    0 \ar[r] & \pi_\ast \mathcal{O}_{\cM} \ar[r] \ar[d] &  \pi_\ast \mathcal{O}_{\cM} \ar[r] \ar[d] & 0 \ar[d] \\
    0\ar[r] & \pi_\ast \diffone_{\cM/S}(L)\ar[r] \ar[d] & \pi_\ast \mathcal{W}_{\cM/S}(L) \ar[r] \ar[d] & \Ker \delta \ar[r] \ar[d] & 0 \\
    0 \ar[r] & 0 \ar[r] & \left( \pi_\ast \mathcal{W}_{\cM/S}(L) \right) / \mathcal{O}_S \ar[r] \ar[d] & \Ker \delta \ar[r] \ar[d] & 0 \\
    & & 0 & 0
  \end{tikzcd}
  \]
  and therefore an isomorphism $ \left( \pi_\ast \mathcal{W}_{\cM/S}(L) \right) / \mathcal{O}_S \to \Ker \delta   $. It remains to show that our hypotheses imply that the image of the morphism
  \[\begin{tikzcd}
  T_S  \ar[r] & T_S \oplus \pi_* \Sym^2 T_{\cM/S} ,
  \qquad
  \theta  \ar[r, mapsto] & (\theta,\rho(\theta))
  \end{tikzcd}
  \]
  is contained in the kernel of the connecting homomorphism $\delta$. In order to do this, let us decompose $\delta=\delta_1 + \delta_2$ into its two components:
   $$\begin{tikzcd}\delta_1: T_S  \ar[r] & R^1\pi_*\mathcal{D}^{(1)}_{\cM/S}(L)\ \ \ \ \textrm{and}\ \ \ \  \delta_2:\pi_*\Sym^2 T_{\cM/S} \ar[r] & R^1\pi_*\mathcal{D}^{(1)}_{\cM/S}(L).\end{tikzcd}$$
   It is then straightforward to check that
   $$R^1\pi_*(\sigma_1)\circ \delta_1= \kappa_{\cM/S}\ \ \ \ \textrm{and}\ \ \ \ R^1\pi_*(\sigma_1)\circ \delta_2= \mu_{L}.$$
  Finally, we observe that $\sigma_1$ induces an injective map
$$\begin{tikzcd}R^1\pi_*(\sigma_1):R^1\pi_*\mathcal{D}^{(1)}_{\cM/S}(L)  \ar[r] & R^1\pi_* T_{\cM/S},\end{tikzcd}$$
as the previous map in the long exact sequence 
$$\begin{tikzcd}\dots\ar[r] & \pi_{*}T_{\cM/S}\ar[r, "{\cup [L]}"] & R^1\pi_*{\cO_{\cM}} \ar[r] & R^1\pi_*\diffone_{\cM/S}(L) \ar[r] & R^1\pi_*T_{\cM/S}\ar[r]& \dots\end{tikzcd} $$
is surjective by condition \ref{vgdj-one}. Thus $(\theta, \rho(\theta))\in \Ker \delta$ 
if and only if $(\kappa_{\cM/S} + \mu_{L} \circ \rho)(\theta)=0$, for any local vector field $\theta$ on $S$.
\end{proof}

\subsection{A flatness criterion} To complete our outline of the general part of the theory, we discuss a general flatness condition for connections constructed via Theorem \ref{vgdj}. It is a verbatim translation of Hitchin's original reasoning \cite[Thm. 4.9]{hitchin:1990} to the algebro-geometric setting, its central ingredient being the requirement that the symbols should Poisson-commute when viewed as homogeneous functions on the relative cotangent bundle.

\begin{theorem}\label{thm_flatness} Under the conditions of Theorem \ref{vgdj} and over a base field of characteristic different from 3, the projective connection constructed from a symbol $\rho$ is projectively flat if
  \begin{enumerate}[(a)]
  \item for all local sections $\theta,\theta'$ of $T_S$,
    \[
     \{ \rho(\theta), \rho(\theta') \}_{T^\ast_{\cM/S}} = 0 ,
     \]
   \item the morphism $\mu_L$ is injective, and
   \item there are no vertical vector fields, $\pi_\ast T_{\cM/S}=0$.
  \end{enumerate}
\end{theorem}
\begin{remark} In the statement and the proof of this theorem we use the fact that the natural morphism
  \[
    \pi_\ast \Sym^k T_{\cM/S} \to \pi_\ast\OO_{T^\ast_{\cM/S}}
  \]
  is an isomorphism of Poisson-algebras onto the weight $k$ part under the natural $\GG_m$-action for $k \leq p-1$; here, the Poisson structure on the left is the one inherited from the commutator bracket on operators of order at most $k$, and the one on the right is the natural one on the cotangent bundle.
\end{remark}
\begin{proof}
  As the connection is defined by projective heat operators (\ref{conn-heat-op}), its flatness is equivalent to the vanishing of the operator
    \begin{equation}\label{comm-flatness}
     [D(\theta),D(\theta')]-D([\theta,\theta']) \in \pi_{e\ast} \left( \diffthree_{\cM/S}(\LL^k)+\difftwo_{\cM}(\LL^k) \right) \big/ \cO_S .
    \end{equation}
    Now it follows from the preceding remark and condition (a) that
    \[
      \sigma_3([D(\theta),D(\theta')]) = \left\{ \sigma_2(D(\theta)),\sigma_2(D(\theta')) \right\}_{T^\ast_{\cM/S}} = \left\{ \rho(\theta),\rho(\theta') \right\}_{T^\ast_{\cM/S}} =  0.
    \]
    Therefore, the operator (\ref{comm-flatness}) is actually at most second order, and we furthermore claim that it really acts only along the fibers of $\cM \rightarrow S$,
    \[
      [D(\theta),D(\theta')]-D([\theta,\theta']) \in \pi_{e\ast} \left( \difftwo_{\cM/S}(\LL^k) \right) \big/ \cO_S .
    \]
    This happens for the same reason the curvature $[\nabla_X,\nabla_Y]-\nabla_{[X,Y]}$ of a connection is of degree zero as a differential operator: one checks (using the subprincipal symbol (\ref{subprincipal})) that (\ref{comm-flatness}) is $\pi^{-1}\cO_S$-linear.
    
    Now we look at the short exact sequence (\ref{sesnoext}), and apply $\pi_\ast$.  As $\mu_L$ is injective by condition (b) and there are no vertical vector fields by (c), we get 
    $$\pi_\ast\difftwo_{\cM/S}(L)\Big/\cO_S \cong \pi_\ast T_{\cM/S} = 0,$$ 
    thus concluding the proof.
\end{proof}

\subsection{The map $\mu_{L}$ }
Finally, we need to get a better understanding of the map $\mu_{L}$ from (\ref{mu-L}), for which we could simply refer to \cite[Cor. 2.4.6]{beilinson.bernstein:1993}. As the proof is not too complicated and uses only a fraction of the machinery of that paper, we thought it worthwile to include it here. We thank an anonymous referee for pointing out considerable simplifications to our previous proof. 
\begin{proposition}\label{thm_mu_O}
In the context outlined above (with $\pi:\cM\rightarrow S$ is a smooth morphism of smooth schemes, and $L$ a line bundle on $\cM$), 
we can write the connecting homomorphism (\ref{mu-L}) as \begin{equation*}
 \mu_{L}= \cup [L] + \cup\left(-\frac{1}{2} [K_{\cM/S}]\right),\end{equation*} where $K_{\cM/S}$ is the relative canonical bundle of $\pi:\cM\rightarrow S$.
\end{proposition}
Note that `half' of this statement ( $\mu_{L}=\cup[L]+\mu_{\mathcal{O}_{\cM}}$) appears in \cite[Lemma 1.16]{welters:1983}, except that Welters uses the extension class of the sheaf of principal parts $\mathcal{P}^{(1)}(L)$ of order $\leq 1$ instead of $\diffone_{\cM/S}(L)$ to define $[L]$ and hence has a minus sign on the right-hand side.  
In a K\"ahler context, with $L$ a polarizing line bundle, the statement of Proposition \ref{thm_mu_O} 
is implied in \cite[p. 364]{hitchin:1990}. In the general complex analytic setting, a Dolbeault-theoretic approach is descibed in \cite[Appendix A.2]{boer:2008}\footnote{The formulas in \cite{beilinson.bernstein:1993} and \cite{boer:2008} are more general expressions that both specialise to the one given in Proposition \ref{thm_mu_O}, but appear different from each other in general.}.
\begin{proof}
  The proof follows from the identification of the opposite of the algebra of differential operators on $L$ with that of $L^{-1}\otimes K_{\cM/S}$ via the adjoint differential operator $D^\circ$, as discussed for example in \cite[1.1.5.(iv)]{beilinson.schechtman:1988}.
  Due to the identity $\mu_{L}=\cup[L]+\mu_{\mathcal{O}_{\cM}}$ observed already by Welters (in arbitrary characteristic), it suffices to show that
  \begin{equation}\label{muadj}
    \mu_{L} = -\mu_{L^{-1}\otimes K}.
  \end{equation}
  For this, consider the adjoint map between sheaves of differential operators $\cA_{\cM/S}(E) \ni D \mapsto D^\circ \in \cA_{\cM/S}(E^\ast \otimes K_{\cM/S})$ defined by the identity
  \[
     \langle e, D^\circ e^\circ \rangle = \langle De, e^\circ \rangle - \LL_{\sigma_1 D}\langle e, e^\circ \rangle ,
  \]
  where $e$ and $e^\circ$ are arbitrary local sections of $E$ and $E^\circ := E^\ast \otimes K_{\cM/S}$, respectively, and $\LL$ is the Lie derivative on the relative canonical bundle. It is straightforward to verify that $D^\circ$ has symbol $\sigma_1(D^\circ) = - \sigma_1(D)$, and that for any regular local function $\phi$
  \[
      (\phi D)^\circ = \phi D^\circ - \langle \sigma_1 D, d_{\cM/S}\phi \rangle ,
  \]
  so that $D\mapsto D^\circ$ is in particular $\pi^{-1}\cO_S$-linear. This zeroth-order deviation from $\cO_{\cM}$-linearity may appear inconvenient at first sight, but it actually permits to extend the adjoint to second order operators, as
  \[
    (\phi D_2)^\circ \circ D_1^\circ = D_2^\circ \circ (\phi D_1)^\circ + (\langle \sigma_1 D_1,d\phi \rangle D_2)^\circ .
  \]
  In this way we obtain a $\pi^{-1}\cO_S$-linear isomorphism of short-exact sequences
  \[
  \begin{tikzcd}[row sep=small]
    0 \ar[r] & \diffone_{\cM/S}(L) \ar[r] \ar[d, swap, "{D \mapsto D^\circ}"] & \difftwo_{\cM/S}(L) \ar[r, "\sigma_2"] \ar[d, swap, "{D \mapsto D^\circ}"] & \Sym^2 T_{\cM/S} \ar[r] \ar[d, "\id"] & 0 \\
    0 \ar[r] & \diffone_{\cM/S}(L^{-1} \otimes K_{\cM/S}) \ar[r] & \difftwo_{\cM/S}(L^{-1} \otimes K_{\cM/S}) \ar[r, "\sigma_2"] & \Sym^2 T_{\cM/S} \ar[r] & 0 ,
  \end{tikzcd}
  \]
  whose push-out along $\sigma_1$ gives
  \[
  \begin{tikzcd}[row sep=small]
    0 \ar[r] & T_{\cM/S} \ar[r] \ar[d, swap, "-\id"] & \difftwo_{\cM/S}(L)/\cO_{\cM} \ar[r, "\sigma_2"] \ar[d, swap, "{D \mapsto D^\circ}"] & \Sym^2 T_{\cM/S} \ar[r] \ar[d, "\id"] & 0 \\
    0 \ar[r] & T_{\cM/S} \ar[r] & \difftwo_{\cM/S}(L^{-1} \otimes K_{\cM/S})/\cO_{\cM} \ar[r, "\sigma_2"] & \Sym^2 T_{\cM/S} \ar[r] & 0 ,
  \end{tikzcd}
  \]
  which proves the necessary identity (\ref{muadj}).
\end{proof}

\begin{remark}
Note that the preceding result remains true in characteristic $p>0$ with $p\neq 2$, since we only use the isomorphism induced by $D \mapsto D^\circ$ between differential operators of order $\leq 2$.
\end{remark}

\section{An algebro-geometric approach to the Hitchin connection for non-abelian theta-functions}\label{mainconstruction}
In this section we construct the Hitchin connection in algebraic geometry.  We want to invoke Theorem \ref{vgdj}, using the symbol $\rho$ from (\ref{rho}) on page \pageref{rho}.  In order to verify that this theorem applies, we need to begin by examining the various ingredients of condition \ref{vgdj-one}.

Note that, compared to the situation of families of abelian varieties (cfr. \cite{welters:1983}, \cite[\S2.3.8]{vangeemen.dejong:1998}), we need a much more detailed knowledge of our candidate symbol, in order to establish flatness of the connection later on (which is done via other means for abelian varieties). 
\subsection{Basic facts about the moduli space of bundles}\label{sect_basicfacts}
At this point we can turn our attention to the particular context we are interested in: the moduli theory of bundles on curves.  In the rest of Section \ref{mainconstruction}, we shall denote by $\pi_s:\calC\rightarrow S$ a smooth family of smooth projective curves of genus $g\geq 2$.  This gives rise, for any integer $r\geq 2$ to a (coarse) relative moduli space of stable bundles of rank $r$ with trivial determinant over the same base, which we shall denote by $\pi_e:\cM\rightarrow S$.  If $g=2$ we will assume that $r\geq 3$.  We shall denote the fibered product by the diagram 
\[\begin{tikzcd}
\calC\times_S \cM\ar[d, "\pi_w"']\ar[r, "\pi_n"] & \cM\ar[d, "\pi_e"] \\ \calC\ar[r, "\pi_s"'] & S
\end{tikzcd}\]
and will simply put $$\pi_c=\pi_e\circ \pi_n=\pi_s\circ \pi_w.$$
Unfortunately $\cM$ is only a coarse moduli space, and a universal bundle over $\calC\times_S \cM$ does not exist (one could argue that it exists over the stack of stable bundles $\mathfrak{M}\rightarrow S$, but does not descend to $\cM$).  Nevertheless, one can speak both of the Atiyah algebroid and Atiyah sequence of the virtual bundle (since these do descend to the coarse moduli space). There exists a unique line bundle $\LL$ over $\cM$, called 
the \it theta line bundle\rm, which is mapped to the relatively ample generator of the relative Picard variety $\Pic(\cM/S)$ (see \cite{drezet.narasimhan:1989,hoffmann:2012}).
In order to avoid making our notations heavier than needed, we shall henceforth pretend a universal bundle $\cE\rightarrow \calC\times_S\cM$ exist. 
Note that this universal bundle is only unique up to tensor product with a line bundle coming from $\cM$. However the trace-free endomorphism bundle $\cE nd^0(\cE)$ is unique.
Similarly the determinant-of-cohomology line bundle on $\cM$ associated to a universal bundle $\cE$, defined as in \cite{MK}
\[
  \lambda(\cE) := \det R^\bullet \pi_{n \ast} (\cE) ,
\]
will depend on the choice of the universal bundle $\cE$. We will use two well-known properties when considering vector bundles with trivial determinant.
\begin{itemize}
\item For any universal bundle $\cE$ and any line bundle $\zeta$ on $\cC \to S$ of degree $g-1$, we have the equality
\cite{drezet.narasimhan:1989,hoffmann:2012}
\begin{equation}\label{thetadet1}
\cL^{-1}=\lambda (\cE \otimes \pi_w^*\zeta).
\end{equation}
\item For any universal bundle $\cE$, we have the equalities \cite{LS}
\begin{equation}\label{thetadet2}
\cL^{-2r} = K_{\cM/S} =  \lambda(\cE nd^0(\cE)).
\end{equation}
\end{itemize}
At various places we shall use the trace pairing $$\begin{tikzcd} \operatorname{Tr}:\End^0(\cE)\times \End^0(\cE)\ar[r]& \cO_{\calC\times_S\cM}\end{tikzcd}$$ to identify $\End^0(\cE)$ with its dual $\End^0(\cE)^*$.

We will need a few other standard facts about the moduli space $\cM$ as well:
\begin{proposition}\label{basicfacts}
We have 
\begin{enumerate}[(a)]
\item $\pi_{n*}\End^0(\cE)=\{0\}$, 
\item $T_{\cM/S}=R^1\pi_{n*}\End^0(\cE)$, 
\item\label{basicfactsthree} $\pi_{e*}T_{\cM/S}=\{0\}$,
 \item\label{basicfactsfour} $R^1\pi_{e*}\cO_{\cM}=\{0\}$.
 \end{enumerate}
 \end{proposition}
The first two of these follow from basic deformation theory.  For the last two, which are also well-known, we include a proof (due to Hitchin) using the Hitchin system in Appendix \ref{appendixbasicfacts}.
\subsection{The Kodaira-Spencer Map}
Our aim in this section is to give a description of the map $$\begin{tikzcd} \Phi: R^1\pi_{s*}T_{\calC/S}\ar[r]& R^1\pi_{e*}T_{\cM/S}\end{tikzcd}$$ (relating deformations of the curve to deformations of the moduli space) which makes the diagram of sheaves on $S$ 
\begin{equation}\label{kappaphi}
\begin{tikzcd}[row sep=-0.5ex, column sep=large]
& R^1\pi_{s*}T_{\calC/S} \ar[dd, "\Phi"] \\ T_S \ar[ur, pos=0.6, "\kappa_{\calC/S}"]\ar[dr, pos=0.7, "\kappa_{\cM/S}" '] & \\ & R^1\pi_{e*}T_{\cM/S}\\
\end{tikzcd}\end{equation}
commute, where $\kappa_{\calC/S}$ and $\kappa_{\cM/S}$ are the Kodaira-Spencer maps, as in (\ref{ks}).  This is a line of reasoning that essentially goes back to Narasimhan and Ramanan \cite{narasimhan.ramanan:1970}.  

On $\calC\times_S\cM$ we have the trace-free relative Atiyah sequence 
\begin{equation}\label{relatseq}\begin{tikzcd} 0\ar[r] & \End^0(\cE) \ar[r] & \atalg^0_{\calC\times_S\cM\big/\cM}(\cE) \ar[r] & T_{\calC\times_S \cM \big/ \cM} \ar[r] & 0.\end{tikzcd}\end{equation}
As we have that $\pi_{n*}\left(T_{C\times_S \cM \big/ \cM}\right)=0$ and $R^2\pi_{n*}\End^0(\cE)=0$, applying $R^1\pi_{n*}$ gives the short exact sequence on $\cM$ \begin{equation}\label{fromsernesi}\begin{tikzcd} 0\ar[r] & R^1\pi_{n*}\End^0(\cE)\ar[r] & R^1\pi_{n*}\atalg^0_{\calC\times_S\cM\big/\cM}(\cE)\ar[r] & R^1\pi_{n*}T_{\calC\times_S \cM\big/\cM}\ar[r] & 0 . \end{tikzcd}\end{equation}

In order to describe the Kodaira-Spencer map $\kappa_{\cM/S}$, we need to start from the short exact sequence $$\begin{tikzcd} 0\ar[r] & T_{\cM/S} \ar[r] & T_{\cM}\ar[r] & \pi_e^*T_S\ar[r] & 0,\\ \end{tikzcd}$$ which is given (see e.g. \cite[\S 3.3.3]{sernesi:2006} for the case of a line bundle -- vector bundles are a straightforward generalisation of the description there, and are discussed in \cite[\S 2.3]{martinengo:2009}) by the pullback of (\ref{fromsernesi}) along the map \begin{equation*}\begin{tikzcd}\pi_e^*\kappa_{\calC/S}:\pi_{e}^*T_S\ar[r]& R^1\pi_{n*}T_{\calC\times_S \cM\big/\cM}\cong \pi_e^*\left(R^1\pi_{s*} T_{\calC/S} \right).\end{tikzcd}\end{equation*}

If we apply $\pi_{e*}$ to this, we obtain finally
\begin{lemma}\label{constrphi} The Kodaira-Spencer map $\kappa_{\cM/S}$ is given by the composition of  $\kappa_{\calC/S}$ with $\Phi$, the connecting homomorphism of (\ref{fromsernesi}): 
\[\begin{tikzcd}[column sep=large, row sep=0ex]
& R^1\pi_{s*}T_{\calC/S}\cong \pi_{e*}\big( R^1\pi_{n*}T_{\calC\times_S \cM\big/ \cM}\big) \ar[dd, shorten=-1ex, "\Phi"]\\ 
T_S\ar[ur, end anchor={[xshift=-4em,yshift=1ex]south}, pos=0.5, "\kappa_{\calC/S}"] \ar[dr, end anchor={[xshift=-4em]north}, pos=0.5, "\kappa_{\cM/S}" '] & \\ 
& R^1\pi_{e*}T_{\cM/S}\cong R^1\pi_{e*}\left(R^1\pi_{n*}\End^0(\cE)\right).
\end{tikzcd}\]
\end{lemma}

\subsection{The Hitchin Symbol}
We have already briefly encountered the Hitchin symbol in (\ref{rho}), we shall clarify the precise definition here in the appropriate relative setting.
We start from the quadratic part of the Hitchin system, relative over $S$, and its associated symmetric bilinear form (temporarily denoted $B$)
\[
  \begin{tikzcd}[column sep=tiny, row sep=small]
    T_{\cM/S}^\ast \arrow[rr, "\operatorname{diag}"] \arrow[rd] && T_{\cM/S}^\ast \otimes T_{\cM/S}^\ast \arrow[ld, "B"]  \\
    & \pi_{n*}K^{\otimes 2}_{\calC\times_S\cM\big/\cM} &
  \end{tikzcd} 
\]
Recall that the bilinear form $B$ is, in the explicit description of the relative cotangent bundle via Higgs fields $T^\ast_{\cM /S} = {\pi_n}_\ast ( \End^0(\cE)\otimes K_{\calC \times_S \cM \big/ \cM} )$, given by the trace
\[
  B(\phi,\psi) = \tr (\phi \circ \psi) .
\]
In particular, it factors further through the symmetric square $\Sym^2 T_{\cM/S}^\ast$. Notice as well that since we assume the characteristic of the base field to be different from 2, the symmetric square is canonically identified with the symmetric 2-tensors, and in particular there is also a canonical identification
\[
  \left( \Sym^2 T_{\cM/S}^\ast \right)^\ast \cong \Sym^2 T_{\cM/S} .
\]
Taking the dual $B^\ast$ of $B$, using Serre duality relative to $\pi_n$ on the domain (where in particular $K_{\calC \times_S \cM / \cM} = \pi_w^\ast K_{\calC / S}$), and pushing down via ${\pi_e}_\ast$ we obtain a map ${\pi_e}_\ast \left( B^\ast \right)$
\[
\begin{tikzcd}
  {\pi_e}_\ast R^1 {\pi_n}_\ast \pi_w^\ast T_{\calC / S}\ar[r, "{\pi_e}_\ast B^\ast"]& {\pi_e}_\ast \Sym^2 T_{\cM / S}.
  \end{tikzcd}
\]
Combining this with flat base change
\[
  R^1 {\pi_n}_\ast \pi_w^\ast T_{\calC / S} \cong \pi_e^\ast R^1 {\pi_s}_\ast T_{\calC / S} ,
\]
we make the following definition.
\begin{definition}\label{hitchinsymbol} The Hitchin symbol $\rho^{\Hit}$ is defined as
\[
\begin{tikzcd}
\rho^{\Hit} := {\pi_e}_\ast \left( B^\ast \right) :
R^1 {\pi_s}_\ast T_{\calC / S} \ar[r]& {\pi_e}_\ast \Sym^2 T_{\cM / S} .
\end{tikzcd}
\]
\end{definition}
The morphism $\rho^{\Hit}$ is in fact an isomorphism.  As we do not need this fact directly, we have relegated it to the Appendix, see Lemma \ref{rho-Hit-isom}.

For our purpose of comparing the symbol map with the Kodaira--Spencer morphism in the general context of Theorem \ref{vgdj}, we need the following alternative description: consider first the surjective evaluation map on $\calC \times_S \cM$:
\begin{equation}\label{eval-end}
\begin{tikzcd}\pi_n^*\pi_{n*}(\End^0(\cE)\otimes \pi_w^*K_{\calC/S}) \ar[r,"\ev"]& \End^0(\cE)\otimes \pi_w^*K_{\calC/S}.
\end{tikzcd}
\end{equation}
Dualizing (\ref{eval-end}) we get a morphism
\[
\begin{tikzcd}
\End^0(\cE)^*\otimes \pi_w^*T_{\calC/S} \ar[r]& \pi_n^*\left(\pi_{n*}\left(\End^0(\cE)\otimes \pi_w^*K_{\calC/S}\right)\right)^*
\end{tikzcd}
\]
so that swapping the first tensor factor and composing with relative Serre duality for $\pi_n$ we obtain a $\mathcal{O}_{\calC\times_S \cM}$-linear morphism
\begin{equation}\label{eval-dual}
\begin{tikzcd}
 \pi_w^*T_{\calC/S} \ar[r,"\ev^\ast"] & \End^0(\cE)\otimes\pi_n^*(R^1\pi_{n*}(\End^0(\cE)^*)).
 \end{tikzcd}
\end{equation}
We also use the trace pairing to identify  $\operatorname{Tr}: \End^0(\cE)\overset{\cong}{\to} \End^0(\cE)^*$.
Now we apply ${\pi_e}_\ast \circ R^1{\pi_n}_\ast$ to  
(\ref{eval-dual}) and, by the isomorphism $R^1\pi_{n*}\End^0(\cE)^*\cong R^1\pi_{n*}\End^0(\cE) \cong T_{\MM/S}$, the projection formula and base change, we obtain a map
\begin{equation}\label{pre-symbol}
\begin{tikzcd}
R^1\pi_{s*}(T_{\calC/S}) \ar[r]& \pi_{e*}\left(T_{\cM/S}\otimes T_{\cM/S}\right).
\end{tikzcd}
\end{equation}
\begin{lemma}\label{defhitchinsymbol}
The map (\ref{pre-symbol}) coincides with the Hitchin symbol \ref{hitchinsymbol}.
\end{lemma}
\begin{proof}
The claimed identity follows from commutativity of the diagram
\[
\begin{tikzcd}[column sep=tiny]
  & R^1 {\pi_n}_\ast \pi_w^\ast T_{\calC /S} \arrow[ld, swap, "R^1{\pi_n}_\ast(\ev^\ast)"] \arrow[rd, "B^\ast"] & \\
  R^1 {\pi_n}_\ast \End^0(\cE) \otimes R^1 {\pi_n}_\ast \left( \End^0(\cE)^\ast \right) \arrow[rr, "\id \otimes (R^1{\pi_n}_\ast \operatorname{Tr}^{-1})^\ast"] & & T_{\cM /S} \otimes T_{\cM /S}.
\end{tikzcd} 
\]
This follows if we in turn dualize, apply Serre duality, for which
\[
\left( R^1 {\pi_n}_\ast (\ev^\ast ) \right)^\ast = {\pi_n}_\ast \left( \ev \otimes \id \right),
\]
(and similarly for the other arrow, where additionally $\operatorname{Tr} = \operatorname{Tr}^\ast$), and observe that the natural pairing on $\End^0(\cE)^\ast \otimes \End^0(\cE)$ coincides with $B \circ (\operatorname{Tr}^{-1} \otimes \id)$ by the definition of $B$ and $\operatorname{Tr}$.
\end{proof}
\subsection{The theta line bundle and its Atiyah algebroid}
Next we need some observations about the Atiyah algebroid of the theta line bundle $\LL$ (see Sect. \ref{sect_basicfacts}).
We recall that $\LL$ is mapped to
the ample generator of $\Pic(\cM/S)$ and that $\LL$ is related to the determinant-of-cohomology line bundle as in (\ref{thetadet1}) and 
(\ref{thetadet2}).
  
  In this setting, the Atiyah sequence for $\LL$ relative to $S$ has a remarkably direct description in terms of the Atiyah sequence of the trace-free relative Atiyah algebroid of $\cE$,
\begin{equation}\label{at-alg}
\begin{tikzcd}
0\ar[r]& \End^0(\cE) \ar[r]& \atalg^0_{\calC\times_S\cM\big/\cM}(\cE) \ar[r]& \pi_w^*T_{\calC/S}\cong T_{\calC\times_S\cM\big/\cM} \ar[r]& 0.
\end{tikzcd}
\end{equation}
Note that, since $\cE nd^0(\cE)$ is uniquely defined, also is $\atalg^0_{\calC\times_S\cM\big/\cM}(\cE)$.
Indeed, we have
\begin{theorem}\label{maintracecompl} The relative Atiyah sequence of the theta line bundle $\cL$ is isomorphic to the first direct image $R^1 \pi_{n \ast}$ of the dual of (\ref{at-alg}):
  \begin{equation}\label{dualR1}
\begin{tikzcd}[column sep=small, row sep=small]
  0 \ar[r]& R^1\pi_{n*}(K_{\cX/\cM})\cong \cO_\cM \ar[r] \ar[d, " \Id_{\cO_M}" swap] & R^1\pi_{n*}\left(\atalg^0_{\cX/\cM}(\cE)^\ast\right) \ar[r] \ar[d, "\cong"] & R^1\pi_{n*}\left(\End^0(\cE)^*\right) \ar[d, "\cong"] \ar[r]& 0 \\
  0 \ar[r]& \cO_\cM \ar[r] & \atalg_{\cM/S}(\LL) \ar[r, "\sigma_1"] & T_{\cM/S} \ar[r]& 0 .
\end{tikzcd}
\end{equation}
\end{theorem}
For a single fixed curve, this result was stated (without proof) in the announcement \cite{ginzburg:1995} (see Theorem 9.1), where it is attributed to Beilinson and Schechtman (even though it does not seem to appear in \cite{beilinson.schechtman:1988}); it can also be derived from results contained in \cite{sun.tsai:2004}.
We give an independent proof in Section \ref{sectionbigproof}.
\subsection{A comment on extensions of line bundles}
Let $X$ be a  scheme, $V$ and $L$ respectively a vector and a line bundle on $X$. Let moreover $F$ be an extension of $L$ by $V$
\[
\begin{tikzcd} 0\ar[r] & V \ar[r,"i"] & F \ar[r, "\pi"] & L \ar[r] & 0.\end{tikzcd}
\]
 By taking the dual and tensoring with $V\otimes L$ we get 
\[
\begin{tikzcd}
0 \ar[r] & V \ar[r] & F^* \otimes V \otimes L \ar[r] & \ar[r] V^*\otimes V\otimes L \ar[r] &0.
\end{tikzcd}\]
Consider now the injective natural map
\begin{eqnarray*}
\psi: L & \to & V^*\otimes V\otimes L \\
\ell & \mapsto & \Id_V \otimes \ell.
\end{eqnarray*}
\begin{lemma}\label{VBremark}
There exists a canonical 
injection $\phi:F\hookrightarrow F^*\otimes V \otimes L$ so that the diagram 
\begin{equation}\label{ext2}
\begin{tikzcd}[row sep=small]
0 \ar[r] & V \ar[d, equal] \ar[r, "i"] & F \ar[d] \ar[d, hookrightarrow, "\phi"] \ar[r, "-\pi "] & L \ar[d, " \psi ", hookrightarrow] \ar[r] & 0 \\
0 \ar[r] & V \ar[r] & F^*\otimes V \otimes L \ar[r] & V^* \otimes V\otimes L \ar[r] & 0
\end{tikzcd} \end{equation}
commutes.
\end{lemma}
\begin{proof}
  We consider the natural $\cO_X$-linear map $\alpha:F \otimes F \to F\otimes L$ defined by
  \[
   \alpha(f_1\otimes f_2) = f_1 \otimes \pi(f_2) - f_2 \otimes \pi(f_1)
  \]
  for local sections $f_1,f_2$ of $F$. Then it is easy to check that the image of $\alpha$ is the subbundle $V\otimes L \subset F\otimes L$. Now the map $\alpha$ naturally corresponds to an $\cO_X$-linear map $\phi: F \to F^\ast \otimes V \otimes L$, which can be described locally in  terms of a basis of local sections $\{e_i\}$ of $F$ and the dual basis $\{e_i^\ast\}$ of $F^\ast$ as
  \[
  \phi(f) = \sum_{i=1}^{\rk F} \left(e_i^\ast \otimes f \otimes \pi(e_i) - e_i^\ast \otimes e_i \otimes \pi(f) \right) .
  \]
  It is now straightforward to check that this $\phi$ makes the above diagram commute.
\end{proof}

\subsection{Locally freeness of $\pi_{e*}(\cL)$}
We will be assuming that the direct image $\pi_{e*}(\cL^k)$ on $S$ is locally free.  In characteristic zero this follows trivially from Kodaira vanishing, but in positive characteristic it is not known in general (but of course it will  always trivally be true for large enough $k$).  For $r=2$, this is however proven in \cite{mehta.ramadas:1996}.  

Note that in characteristic zero, a coherent sheaf with a flat projective connection will necessarily be locally free, but this need not be true in general.

\subsection{The relation between $\rho^{\Hit}, \Phi$, and $\LL$}
We can now state the final ingredient we will need to prove the existence of the Hitchin connection:
\begin{proposition}\label{phi-rho-L}
The sheaf morphism $\Phi$ from (\ref{kappaphi}) equals minus the composition $(\cup [\LL]) \circ\rho^{\Hit}$ of the Hitchin symbol and the characteristic class $[\LL]$, i.e. the following diagram of sheaves on $S$ commutes:
\[
\begin{tikzcd}[row sep=small]
R^1\pi_{s*} T_{\mathcal{C}/S}  \ar[rr, "-\Phi"] \ar[dr, "\rho^{\Hit} "'] & & R^1\pi_{e*}T_{\cM/S}.\\
& \pi_{e*}\Sym^2 T_{\cM/S} \ar[ur,  "\cup {[\LL]} "'] &
\end{tikzcd}
\]
\end{proposition}
\begin{proof}
We begin with the trace-free Atiyah sequence on $\calC\times_S\cM$ for $\cE$, relative to $\pi_n$, as introduced in Section \ref{seqandcon}.  To keep the notation light, we shall denote in this proof the Atiyah algebroid $\atalg^0_{\calC\times_S \cM\big/\cM}(\cE)$ simply by $\atalg$.  By using the evaluation maps, as in (\ref{eval-end}), dualizing, and tensoring  with $\pi^*_wT_{\calC/S}\otimes \End^0(\cE)$,  we obtain the following natural map of exact sequences:
\begin{equation}\label{doubleseq}
\begin{tikzcd}[column sep=small, row sep=small] 
0 \ar[r]  & \End^0(\cE)   \ar[r]\ar[d,equal]  & \begin{array}{@{}c@{}} \End^0(\cE)\otimes \\ \atalg^*\otimes \pi_w^*T_{\calC/S} \end{array} \ar[r] \ar[d] & \begin{array}{@{}c@{}} \End^0(\cE)\ \otimes\\ \End^0(\cE)^*\otimes \pi_w^*T_{\calC/S}\end{array} \ar[r] \ar[d] & 0 \\
0 \ar[r]  &  \End^0(\cE) \ar[r] & \begin{array}{@{}c@{}} \End^0(\cE)\ \otimes\\ \pi_n^*(\pi_{n*}(\atalg\otimes \pi_w^* K_{\calC/S}))^*\end{array} \ar[r] & \begin{array}{@{}c@{}}\End^0(\cE)\ \otimes\\ \pi_n^*(\pi_{n*}(\End^0(\cE)\otimes \pi_w^*K_{\calC/S}))^*\end{array} \ar[r]  & 0.
\end{tikzcd}
\end{equation}
By relative Serre duality for $\pi_n$, the lower exact sequence is equal to the following
\begin{equation}\label{relserdual}
\begin{tikzcd}
0\ar[r] & \begin{array}{@{}c@{}}\End^0(\cE)\ \otimes \\ \pi_n^*(R^1\pi_{n*}\pi_w^*K_{\calC/S})\end{array}\ar[r] &\begin{array}{@{}c@{}} \End^0(\cE)\ \otimes\\ \pi_n^*(R^1\pi_{n*}\atalg^*)\end{array} \ar[r] & \begin{array}{@{}c@{}} \End^0(\cE)\ \otimes\\ \pi_n^*(R^1\pi_{n*}\End^0(\cE)^*)\end{array} \ar[r]& 0.
\end{tikzcd}
\end{equation}
By plugging $V=\End^0(\cE)$, $L=\pi_w^*T_{\calC/S}$ and $F=\atalg$ in Lemma \ref{VBremark}, we get a map of exact sequences
\begin{equation}\label{finalmap}
\begin{tikzcd}[row sep=small]
0 \ar[r] & \End^0(\cE) \ar[d, equal] \ar[r] &  \atalg \ar[r] \ar[d] & \pi_w^* T_{\calC/S} \ar[d] \ar[r] & 0 \\
0 \ar[r] & \End^0(\cE) \ar[r] &\begin{array}{@{}c@{}} \End^0(\cE)\ \otimes\\ \atalg^*\otimes \pi_w^*T_{\calC/S}\end{array} \ar[r] & \begin{array}{@{}c@{}} \End^0(\cE)\ \otimes\\ \End^0(\cE)^*\otimes \pi_w^*T_{\calC/S} \end{array} \ar[r] & 0. 
\end{tikzcd}
\end{equation}
Hence, by composing the short exact sequence maps (\ref{finalmap}) and (\ref{doubleseq}), and using the isomorphism of the target exact sequence with that of (\ref{relserdual}), we get a new map of exact sequences:
\begin{equation}\label{comm-at-ks}
\begin{tikzcd}[column sep=scriptsize, row sep=small]
0 \ar[r] & \End^0(\cE) \ar[d, equal] \ar[r]  & \atalg \ar[r] \ar[d] & \pi_w^*T_{\calC/S} \ar[r] \ar[d] & 0 \\
0 \ar[r] & \begin{array}{@{}c@{}} \End^0(\cE)\ \otimes\\ \pi_n^*(R^1\pi_{n*}(\pi_w^*K_{\calC/S}))\end{array} \ar[r] & \begin{array}{@{}c@{}} \End^0(\cE)\ \otimes \\ \pi_n^*(R^1\pi_{n*}\atalg^*) \end{array} \ar[r] & \begin{array}{@{}c@{}} \End^0(\cE)\ \otimes \\ \pi_n^*(R^1\pi_{n*}(\End^0(\cE)^*)) \end{array} \ar[r] & 0.
\end{tikzcd}
\end{equation}
 By taking the direct image $R^1\pi_{n*}$ of both sequences, they remain exact and we obtain the commutative diagram
\begin{equation}\label{R1comm-at-ks}
\begin{tikzcd}[row sep=small]
0 \ar[r] & R^1\pi_{n*}\End^0(\cE) \ar[d, equal] \ar[r] &  R^1\pi_{n*}\atalg \ar[r] \ar[d] & R^1\pi_{n*}\pi_w^*T_{\calC/S} \ar[d] \ar[r] & 0 \\
0 \ar[r] & R^1\pi_{n*}\End^0(\cE)\ar[r] & \begin{array}{@{}c@{}} R^1\pi_{n*}\End^0(\cE)\\ \otimes\ (R^1\pi_{n*}\atalg^*)\end{array} \ar[r] & \begin{array}{@{}c@{}}R^1\pi_{n*}\End^0(\cE)\ \otimes\\ (R^1\pi_{n*}(\End^0(\cE)^*))\end{array} \ar[r]  & 0.
\end{tikzcd}
\end{equation}
We now apply $\pi_{e*}$ to both exact sequences in (\ref{R1comm-at-ks}).   The claimed equality is proven once we consider the commutative diagram given by the connecting homomorphisms:
\begin{equation}\label{hitcupL}
\begin{tikzcd}[row sep=small]
 R^1\pi_{s*}(T_{\calC/S}) \ar[r, "-\Phi"] 
 \ar[d, swap, "{\rho^{\Hit}}"]
  & R^1\pi_{e*}(T_{\cM/S}) \ar[d, equal]  \\
  \pi_{e*}(T_{\cM/S}\otimes T_{\cM/S}) \ar[r,"{\cup [\LL]} "]  &  R^1\pi_{e*}(T_{\cM/S}). 
\end{tikzcd}
\end{equation}
Since the bottom row of (\ref{R1comm-at-ks}) is given by tensoring (\ref{dualR1}) by $R^1\pi_{n*}\End^0(\cE)$, by Theorem \ref{maintracecompl} the connecting homomorphism for the bottom row is given by the relative Atiyah class of $\LL$.  By Lemma \ref{defhitchinsymbol}, the left vertical map is given by the Hitchin symbol $\rho^{\Hit}$. Since the upper exact sequence of (\ref{comm-at-ks}) is the same as the sequence (\ref{relatseq}) but with one sign changed (as in (\ref{ext2})), by Lemma \ref{constrphi} the connecting homomorphism for the top row of (\ref{hitcupL}) is given by $-\Phi$.  
\end{proof}

\subsection{Existence and flatness of the connection}
We can now summarize the algebro-geometric construction of the Hitchin connection:
\begin{theorem}
\label{existenceconnection} 
Let $k$ be a positive integer.
Suppose a smooth family 
$\pi_{e}:\calC\rightarrow S$ of projective curves of genus $g\geq 2$ (and $g\geq 3$ if $r=2$) is given as before, defined over an algebraically closed field of characteristic different from $2$, not dividing $r$ and $k+r$, and such that $\pi_{e*}(\LL^k)$ is locally free. Then there exists a unique projective connection on the vector bundle $\pi_{e*}(\LL^{k})$ of non-abelian theta functions of level $k$, 
induced by a heat operator with symbol $$\rho=\frac{1}{r+k}\,\left(\rho^{\Hit}\circ \kappa_{\calC/S}\right).$$
\end{theorem}
\begin{proof} We establish the existence of the projective connection by invoking Theorem~\ref{vgdj} for the line bundle $\LL^{k}$ over $\cM$. 
We recall from (\ref{thetadet2}) the equality $K_{\cM/S}=\LL^{-2r}$.
From Proposition \ref{thm_mu_O} we therefore have that $$\mu_{\LL^k}=\cup(r+k)[\LL],$$ and hence (using Proposition \ref{phi-rho-L} and (\ref{kappaphi})) we have   $$\mu_{\LL^k}\circ\rho=\mu_{\LL^k}\circ \frac{1}{r+k}\,\left(\rho^{\Hit}\circ \kappa_{\calC/S}\right)=\left(\cup[\LL]\right)\circ \rho^{\Hit}\circ \kappa_{\calC/S}=-\Phi\circ \kappa_{\calC/S} = -\kappa_{\cM/S},$$ which 
establishes condition  \ref{vgdj-one} of Theorem~\ref{vgdj}.
Condition \ref{vgdj-two} is trivially satisfied because of Proposition \ref{basicfacts}, and
condition \ref{vgdj-three} follows from the algebraic Hartogs's theorem \cite[Lemma 11.3.11]{vakil:2017}, together with the well-known fact that the relative coarse moduli space  $\cM^{\operatorname{ss}}$ of semi-stable bundles with trivial determinant (which is singular but normal) is proper over $S$, and if $g>2$ or $r>2$, 
the complement of $\cM$ will have codimension greater than one in $\cM^{\operatorname{ss}}$.
\end{proof}
As for the curvature of the connection, we have:
\begin{theorem}\label{connection-flat}
Suppose furthermore that the characteristic of the base field is different from 3. Then the projective connection constructed in Theorem \ref{existenceconnection} is flat.
\end{theorem}
\begin{proof}
  We apply Theorem \ref{thm_flatness}: condition (a) holds since by definition of the Hitchin symbol the corresponding homogeneous functions on $T^\ast_{\cM/S}$ are the quadratic components of the Hitchin system, and hence Poisson-commute,
  \[
  \left\{ \rho^{\Hit}(\theta),\rho^{\Hit}(\theta') \right\}_{T^\ast_{\cM/S}} =  0 .
  \]
  Condition (b) is satisfied as $\mu_{\LL^k}$ is injective (see Lemma \ref{mu-L-inj} in Appendix \ref{appendixbasicfacts}), and (c) holds by Proposition \ref{basicfacts}.
\end{proof}

\section{Proof of Theorem \ref{maintracecompl}}\label{sectionbigproof}
We shall need the theory of the \emph{trace complex}, due to Beilinson and Schechtman, or rather a variation thereon due to Bloch and Esnault -- see \cite{beilinson.schechtman:1988} and \cite{bloch.esnault:2002}.  In Appendix \ref{appendixtracecomplex} a summary of this theory is given, and we refer to it  for definitions of the complexes $\tensor*[^{tr\!\!}]{\cA}{^\bullet}$, $\cB^\bullet$, and $\tensor*[^0]{\cB}{^\bullet}$.  We will be applying the trace complex in our particular setting here, where $\cM$ is as in Section \ref{sect_basicfacts}, $\cX = \calC \times_S \cM$ and $f = \pi_n$.  In this context we find that the trace complex simplifies significantly, to give Theorem \ref{maintracecompl}.

Before proving Theorem \ref{maintracecompl} we need to prove a few auxiliary results.
\begin{lemma}\label{tecfacts}
Following the above notation:
\begin{enumerate}[(a)]
\item the direct image $\pi_{n*}{}^0\cB^0(\cE)$ equals 0;
\item the natural map $R^1\pi_{n*}\End^0 (\cE) \to 
R^1\pi_{n*}{}^0\cB^0(\cE)$ is zero.
\end{enumerate}
\end{lemma}
\begin{proof}
Recall from Section \ref{rmk_tracelessB} that we have a short exact sequence
$$
\begin{tikzcd}0 \ar[r]& \cE nd^0(\cE) \ar[r]& {}^0\cB^0(\cE) \ar[r]& \pi_n^{-1}T_{\cM/S} \ar[r]& 0.\end{tikzcd}$$
By applying the direct image $\pi_{n*}$ we get
$$\begin{tikzcd}[column sep=small] 0 \ar[r]& \pi_{n*}\cE nd^0(\cE) \ar[r]& \pi_{n*}{}^0\cB^0(\cE) \ar[r]& T_{\cM/S} \ar[r]& R^1\pi_{n*}\cE nd^0(\cE) \ar[r]& R^1\pi_{n*}{}^0\cB^0(\cE) \ar[r]& \cdots .\end{tikzcd}$$
Now, by Proposition \ref{basicfacts} $(a)$ and $(b)$, $\pi_{n*}\cE nd^0(\cE)=0$  and the map $T_{\cM/S} \to R^1\pi_{n*}\cE nd^0(\cE)$ is an isomorphism. The two claims follow.
\end{proof}
\begin{proposition}\label{isodirimage}
There exists an isomorphism $\phi: R^1\pi_{n*}{}^0\cB^{-1}(\cE) \to R^0\pi_{n*}\cB^\bullet(\cE nd^0(\cE))$ that makes the following diagram commute.
\[
\begin{tikzcd}[column sep=small, row sep=small]
  0 \ar[r]& R^1\pi_{n*}(K_{\cX/\cM}) \cong \cO_\cM \ar[r] \ar[d, "2r\cdot \Id_{\cO_\cM}"] & R^1\pi_{n*} {}^0\cB^{-1}(\cE) \ar[r] \ar[d, "\phi", "\cong" swap] & R^1\pi_{n*}\left(\End^0(\cE)\right)\cong T_{\cM/S} \ar[d, "\cong"] \ar[r]& 0 \\
  0 \ar[r] & R^0\pi_{n*}K_{\cX/\cM}[1] \cong \cO_\cM \ar[r] & R^0\pi_{n*}\cB^\bullet(\cE nd^0(\cE)) \ar[r] & T_{\cM/S} \ar[r] & 0 .
\end{tikzcd}\]
In particular $\phi$ induces $2r\cdot \Id_{\cO_\cM}$ on $\cO_\cM$.
\end{proposition}
This Proposition is already proved by combining \cite[Thm. 3.7 and Cor. 3.12]{sun.tsai:2004}. For the sake of self-containedness, here we give a complete but slightly different proof of this statement.
\begin{proof}
We construct $\phi$ in several steps, notably as the composition of three maps. First of all, let us define a map
$$\begin{tikzcd}\phi_1: R^1\pi_{n*}{}^0\cB^{-1}(\cE) \ar[r]& R^0\pi_{n*}{}^0\cB^\bullet(\cE).\end{tikzcd}$$
For the sake of clarity, we recall the definition of the $0^{th}$ direct image
$R^0\pi_{n*}{}^0\cB^\bullet(\cE).$ We choose an acyclic resolution of the complex ${}^0\cB^\bullet(\cE)$ as follows
\[
\begin{tikzcd}[row sep=small]
{}^0\cB^{-1}(\cE) \ar[r]\ar[d, hook]& {}^0\cB^0(\cE) \arrow[d, hook] \\
\cC^0({}^0\cB^{-1}(\cE)) \ar[r, "\delta^0"]\ar[d, two heads] & \cC^0({}^0\cB^0(\cE))\ar[d, two heads] \\
\cC^1({}^0\cB^{-1}(\cE)) \ar[r, "\delta^1"] & \cC^1({}^0\cB^0(\cE))
\end{tikzcd}\]
We push this diagram forward through $\pi_n$ and consider the following one:
\[
\begin{tikzcd}[row sep=small]
\pi_{n*}\cC^0({}^0\cB^{-1}(\cE)) \ar[r, "\delta^0"]\ar[d, "d_{-1}"] & \pi_{n*}\cC^0({}^0\cB^0(\cE))\ar[d, "d_0"] \\
\pi_{n*}\cC^1({}^0\cB^{-1}(\cE)) \ar[d, two heads] \ar[r, "\delta^1"] & \pi_{n*}\cC^1({}^0\cB^0(\cE)) \ar[d, two heads] \\
R^1\pi_{n*}{}^0\cB^{-1}(\cE)\ar[r] & R^1\pi_{n*}{}^0\cB^{0}(\cE)
\end{tikzcd}\]
Remark that the lower horizontal arrow factors as $$R^1\pi_{n*}{}^0\cB^{-1}(\cE)  \to R^1\pi_{n*} \End^0(\cE) \to R^1\pi_{n*}{}^0\cB^{0}(\cE).$$ 
By definition we have that $R^0\pi_{n*}{}^0\cB^{-1}(\cE):= \Ker(B)/ \Ima(A)$, where
\begin{equation*}
\begin{tikzcd}[row sep=0pt]
\pi_{n*}\cC^0({}^0\cB^{-1}(\cE)) \ar[r, "A"]& \pi_{n*}\cC^0({}^0\cB^0(\cE)) \oplus \pi_{n*}\cC^1({}^0\cB^{-1}(\cE))\ar[r, "B"] & \pi_{n*} \cC^1({}^0\cB^0(\cE)) \\
(\gamma) \ar[r, mapsto]& (\delta^0(\gamma),d_{-1}(\gamma)) & \\
 & (\alpha, \beta) \ar[r, mapsto] & d_0(\alpha) - \delta^1(\beta).\\
\end{tikzcd}
\end{equation*}
Hence we can define a map
\begin{eqnarray*}
\tilde{\phi}: \pi_{n*}\cC^1({}^0\cB^{-1}(\cE)) & \to & \Ker(B);\\
\beta & \mapsto & (\alpha, \beta);
\end{eqnarray*}
where $\alpha\in\pi_{n*}\cC^0({}^0\cB^0(\cE))$ is uniquely defined by the formula $d_0(\alpha)=\delta^1(\beta)$. In fact we observe that Lemma \ref{tecfacts} implies that $d_0$ is injective and that $\Ima(\delta^1) \subseteq \Ima(d^0)$. The map $\tilde{\phi}$ descends to the first of our three maps:
\begin{eqnarray*}
\phi_1: R^1\pi_{n*}{}^0\cB^{-1}(\cE) & \to & R^0\pi_{n*}{}^0\cB^\bullet(\cE);\\
\bar{\beta} & \mapsto & \overline{(\alpha,\beta)};
\end{eqnarray*}
where the overline should be intended as just taking the corresponding classes.

The second map is defined as follows (see App. \ref{appendixsplitting} for the precise definitions of $\widehat{\ad}$ and $\widetilde{\ad}$):
\begin{eqnarray*}
\phi_2: R^0\pi_{n*}{}^0\cB^\bullet(\cE) & \to & R^0\pi_{n*}({}^0\cB^{-1}(\End^0(\cE)) \to \cB^0(\End^0(\cE)));\\
\overline{(\alpha,\beta)} & \mapsto & (\widetilde{\ad}(\alpha), \widehat{\ad}(\beta));
\end{eqnarray*}
where we abuse once more of the notation (and of the reader's patience) by denoting by $\widehat{\ad}$ and $\widetilde{\ad}$ also the maps on the direct images. Note also that here we consider $\widetilde{\ad}$ as defined on the quotient ${}^0\cB^0(\cE)$ of the subsheaf $\cB^0(\cE)\subset \mathcal{A}(\cE)$, and we are allowed to do so since the trivial sheaf is in $\Ker(\widetilde{\ad})$. Moreover, we can consider $\cB^0(\End^0(\cE))$ as the target space of $\widetilde{\ad}$ the image of ${}^0\cB^0(\cE)$ via $\widetilde{\ad}$ is contained in $\cB^0(\End^0(\cE))\subset \cA(\End^0(\cE))$.

The third map is induced on $R^0\pi_{n*}({}^0\cB^{-1}(\End^0(\cE)) \to \cB^0(\End(\cE)))$ by the natural inclusion ${}^0\cB^{-1}(\End^0(\cE)) \hookrightarrow \cB^{-1}(\End(\cE))$. Hence this gives a natural map
$$\begin{tikzcd}\phi_3: R^0\pi_{n*}({}^0\cB^{-1}(\End^0(\cE)) \ar[r]& \cB^0(\End^0(\cE))) \ar[r]& R^0 \pi_{n*}\cB^\bullet(\End^0(\cE)).\end{tikzcd}$$
It is a standard check that these three maps are well defined and pass to the quotient in cohomology.

The situation is now the following, we have two exact sequences and a map $\phi:= \phi_3\circ \phi_2 \circ \phi_1$ between extensions:
\begin{equation*}
\begin{tikzcd}[column sep=0.8em, row sep=small]
0 \ar[r] &[-0.5ex] R^1\pi_{n*}(K_{\cX/\cM}) \cong \cO_\cM \ar[r]\ar[d]& R^1\pi_{n*}{}^0\cB^{-1}(\cE) \ar[d, "\phi"]\ar[r]& R^1\pi_{n*}(\End^0(\cE))\cong T_{\cM/S} \ar[r]\ar[d]&[-0.5ex] 0\\
0 \ar[r] &[-0.5ex] R^0\pi_{n*}(K_{\cX/\cM})[1])\cong \cO_{\cM} \ar[r]& R^0\pi_{n*}\cB^\bullet(\End^0(\cE)) \ar[r]& T_{\cM/S} \ar[r] &[-0.5ex] 0.
\end{tikzcd}
\end{equation*}
Now, suppose we have a class $\bar{\beta}$ in $R^1\pi_{n*}{}^0\cB^{-1}(\cE)$, and let us consider $\beta$ a local section of $\pi_{n*}\cC^1({}^0\cB^{-1}(\cE))$ representing $\bar{\beta}$. If we denote as above by $\alpha\in \pi_{n*}\cC^0({}^0\cB^0(\cE))$ the uniquely defined local section as in the definition of $\tilde{\phi}$, then $\phi$ sends $\beta$ on $\overline{(\widetilde{\ad}(\alpha),\widehat{\ad}(\beta))}$.

By Proposition \ref{ST310} we have a commutative diagram
\begin{equation*}
\begin{tikzcd}[row sep=small]
0 \ar[r] & K_{\cX/\cM} \ar[r]\ar[d, "\cdot 2r"]& {}^0\cB^{-1}(\cE) \ar[d, "\widehat{\ad}"]\ar[r]& \End^0(\cE) \ar[r]\ar[d, "\ad_0"]& 0\\
0 \ar[r] & K_{\cX/\cM}\ar[r]& {}^0\cB^{-1}(\End^0(\cE)) \ar[r]& \End^0(\End^0(\cE)) \ar[r] & 0.
\end{tikzcd}
\end{equation*}
which implies the claim about the restriction of $\phi$ to $\cO_\cM$. Thus $\phi$ also descends to a $\cO_\cX$-linear map $\phi^T: T_{\cM/S} \to T_{\cM/S}.$ Remark in fact that, again by Appendix \ref{appendixsplitting} and the observations on $\widetilde{\ad}$ made here above, $\phi^T$ is induced by the adjoint map between the following exact sequences.
\begin{equation*}
\begin{tikzcd}[row sep=small]
0 \ar[r] & \End^0(\cE) \ar[r]\ar[d, "\ad"]& {}^0\cB^{0}(\cE) \ar[d, "\widetilde{\ad}_0"]\ar[r]& \pi_{n}^{-1}(T_{\cM/S}) \ar[r]\ar[d, "\Id", "\cong" swap]& 0\\
0 \ar[r] & \End(\End^0(\cE)) \ar[r]& \cB^{0}(\End^0(\cE)) \ar[r]& \pi_{n}^{-1}(T_{\cM/S}) \ar[r] & 0.
\end{tikzcd}
\end{equation*}
\end{proof}
\begin{proof}[Proof of Theorem \ref{maintracecompl}]The isomorphism of exact sequences claimed in the theorem will follow by composing the following isomorphisms. In the diagram below they will be composed vertically from the first to the fifth. First we apply $R^1\pi_{n*}$ to the second identification from Theorem \ref{thmdualityB}. Then we compose with the map from Proposition \ref{isodirimage}. The third map is the isomorphism from Theorem \ref{easyBS} applied to $\cE nd^0(\cE)$ (recall that $\lambda(\cE nd^0(\cE)) = \cL^{-2r}$). The fourth and fifth map is the canonical isomorphism $\cA(\cL^{-1}) \cong \cA(\cL^{-2r})$ obtained by scaling appropriately the extension as in Lemma \ref{extens} with $k=2r$ and $L=\LL^{-1}$. Finally the last vertical isomorpism $\cA(\cL^{-1}) \to \cA(\cL)$ is the canonical map
between the Atiyah algebra of $\cL^{-1}$ and its dual $\cL$ (with the opposite symbol map).
Hence we obtain the following commutative diagram
\begin{equation*}
\begin{tikzcd}[column sep=small, row sep=small]
0 \ar[r]& \cO_{\cM} \ar[r]\ar[d, "\cong", "\Id_{\cO_{\cM}}" swap] & R^1\pi_{n*}(\cA^0_{\cX/\cM}(\cE)^\ast) \ar[r]\ar[d, "\cong", "\widetilde{Res}" swap] & R^1\pi_{n*}(\cE nd^0(\cE)^*) \ar[d, "\cong", "-\operatorname{Tr}" swap]\ar[r]& 0 \\
0 \ar[r]& \cO_{\cM} \ar[d, "\cong", "2r\cdot \Id_{\cO_{\cM}}" swap] \ar[r]& R^1\pi_{n*}({}^0\cB^{-1}(\cE)) \ar[r]\ar[d, "\phi"]& R^1\pi_{n*}(\cE nd^0(\cE)) \ar[r]\ar[d, "\cong"]& 0 \\
0 \ar[r]& \cO_{\cM} \ar[d, "\cong"] \ar[r]& R^1\pi_{n*}\cB^{\bullet}(\cE nd^0(\cE)) \ar[r]\ar[d, "\cong"]& T_{\cM/S} \ar[r]\ar[d, "\cong"]& 0 \\
0 \ar[r]& \cO_{\cM} \ar[d, "\cong"]\ar[r]& \cA(\cL^{-2r}) \ar[r, "\sigma_1"]\ar[d, "\cong"]& T_{\cM/S} \ar[r]\ar[d, "\cong"]& 0 \\
0 \ar[r]& \cO_{\cM} \ar[d, "\cong"] \ar[r, "\frac{1}{2r}"]& \cA(\cL^{-1}) \ar[d, "\cong"]  \ar[r, "\sigma_1"]& T_{\cM/S} \ar[d, "\cong"] \ar[r] & 0 \\
0 \ar[r]& \cO_{\cM} \ar[r, "\frac{1}{2r}"]& \cA(\cL) \ar[r, "-\sigma_1"]& T_{\cM/S} \ar[r] & 0.
\end{tikzcd}
\end{equation*}
Note that the first vertical right hand side map is $-\operatorname{Tr}$. This means that the extension class defining the upper short exact sequence is
equal to the standard Atiyah sequence of $\cL$ as claimed in the Theorem. 
\end{proof}

\appendix
\section{The trace complex, following Beilinson--Schechtman and Bloch--Esnault}\label{appendixtracecomplex}
We give here a  presentation of the parts of the theory of \emph{trace complexes} (due to Beilinson and Schechtman \cite[\S 2]{beilinson.schechtman:1988}, see also \cite{esnault.tsai:2000}) that we need.  We then describe an alternative approach to the trace complexes, suggested by Bloch and Esnault \cite[\S 5.2]{bloch.esnault:2002}.
 
 In fact, to suit our purposes, we make two minor variations: first, we make some small changes to ensure that the construction works in positive characteristic (apart from 2), and secondly, we phrase everything in a relative context.  The latter is trivial on a technical level, but we do it as the Bloch-Esnault approach requires an extra condition, which, when we invoke it in the main part of the article, is only satisfied in a relative setting.

Section \ref{sectiononBS} below covers the original trace complex, and is just expository.  In Section \ref{sectiononBE}, where the alternative of Bloch-Esnault is explained, we also give proofs for various assertions merely stated in \cite{bloch.esnault:2002}.

For the purpose of this appendix, we consider a family of smooth projective curves $f: \cX \to \cM$ of genus $g\geq 2$, relative to a smooth base scheme $S$,
\[
 \begin{tikzcd}[column sep=small, row sep=small]
  \cX \arrow[r, "f"] \arrow[rd] &  \cM \arrow[d] \\
  & S , 
 \end{tikzcd}
\]
together with a vector bundle $\cE \to \cX$.  We shall write $\cE^\circ$ for $\cE^\ast \otimes K_{\cX/\cM}$.

 The trace complex we are interested in describes the Atiyah algebroid $\cA_{\cM/S}(\det R^\bullet f_\ast \cE)$ (remark that our notation differs from Beilinson and Schechtman's: our $\cM$ is their $S$, and  our $S$ is just a point in  \cite{beilinson.schechtman:1988}).
\subsection{The Beilinson--Schechtman trace complex $\tensor*[^{\operatorname{tr}\!\! }]{\cA}{^\bullet}(\cE)$}\label{sectiononBS} 
\newcommand{\fp}{\cC\times_S \cM}
\subsubsection{Overview}
The relative tangent bundle $T_{\cX/S}$ contains as subsheaves $T_{\cX/\cM} \subset T_{f/S} \subset T_{\cX/S}$, where (with $df: T_{\cX/S}\rightarrow f^*T_{\cM/S}$)
\[
 T_{f/S} := (df)^{-1} f^{-1}T_{\cM/S} ,
\]
and corresponding Atiyah algebroids
\begin{equation*}
 \cA_{\cX/\cM}(\cE) \hookrightarrow \cA_{f/S}(\cE) \hookrightarrow \cA_{\cX/S}(\cE) .
\end{equation*}
The Beilinson-Schechtman trace complex is a three-term complex \[
\tensor*[^{\operatorname{tr}\!\! }]{\cA}{^\bullet}(\cE)=\left\{\begin{tikzcd} \tensor*[^{\operatorname{tr}\!\! }]{\cA}{^{-2}}(\cE) \ar[r]& \tensor*[^{\operatorname{tr}\!\! }]{\cA}{^{-1}}(\cE) \ar[r] & \tensor*[^{\operatorname{tr}\!\! }]{\cA}{^0}(\cE) \end{tikzcd}\right\},
\]
where $\tensor*[^{\operatorname{tr}\!\! }]{\cA}{^{-2}}(\cE)=\cO_{\cX}$, $\tensor*[^{\operatorname{tr}\!\! }]{\cA}{^0}(\cE) =\cA_{f/S}(\cE)$, and $\tensor*[^{\operatorname{tr}\!\! }]{\cA}{^{-1}}(\cE)$ is an extension (to be defined below in Section \ref{BS-construction})
)
\begin{equation}\label{at1}
\begin{tikzcd}
0\ar[r]& K_{\cX/\cM} \ar[r] & \tensor*[^{\operatorname{tr}\!\!}]{\cA}{^{-1}}(\cE) 
\ar[r, "\operatorname{res}"] & 
\cA_{\cX/\cM}(\cE) \ar[r]& 0, 
\end{tikzcd}
\end{equation}
which fits into the following commutative diagram
\begin{equation}\label{3cpxes}
\begin{tikzcd}[row sep=small] 
& \mathcal{O}_{\cX} \ar[r, equal] \ar[d, "d_{\cX/\cM}"] 
& \tensor*[^{\operatorname{tr}\!\! }]{\cA}{^{-2}}(\cE) \ar[d, "{d_{\cX/\cM}}"] 
&        &  \\
0 \ar[r]   & K_{\cX/\cM} \ar[r]                                 & \tensor*[^{\operatorname{tr}\!\! }]{\cA}{^{-1}}(\cE) \ar[r, "\operatorname{res}"] \ar[d, "\operatorname{res}"]& \cA_{\cX/\cM}(\cE) \ar[r] \ar[d] & 0 \\
           &                                                  & \tensor*[^{\operatorname{tr}\!\! }]{\cA}{^0}(\cE) \ar[r, equal] &  \cA_{f/S}(\cE). & 
           \end{tikzcd}
\end{equation}
The main use of the trace complex $\tensor*[^{\operatorname{tr}\!\!} ]{\cA}{^\bullet}(\cE)$ is the following:
\begin{theorem}[{\cite[Thm. 2.3.1]{beilinson.schechtman:1988}} ]\label{main1}
The relative Atiyah sequence of the determinant-of-cohomology line bundle 
$$ \lambda(\cE) = \det R^\bullet f_\ast \cE := \det f_\ast \cE\otimes \left(\det R^1 f_\ast \cE\right)^*$$ 
of $\cE$ with respect to $f$ is canonically isomorphic to the short exact sequence
\[
\begin{tikzcd}[column sep=tiny, row sep=small]
0\ar[r] & R^0f_*\left(\Omega^{\bullet}_{\cX/\cM}[2]\right) \ar[r]\ar[d,"\cong"]& R^0 f_\ast(\tensor*[^{\operatorname{tr}\!\! }]{\cA}{^\bullet}(\cE)) \ar[r] \ar[d, "\cong"] & R^0f_*\left(\left({
{\begin{array}{@{}c@{}} {\atalg_{\cX/\cM}(\cE)}
\\ \downarrow \\ \atalg_{f/S}(\cE) \end{array}}
}\right)[1]\right) \ar[r]\ar[d,"\cong"] & 0\\
0 \ar[r] & \cO_{\cM} \ar[r] & \atalg_{\cM/S}(\lambda(\cE)) \ar[r, "\sigma_1"] \ar[r] & T_{\cM/S} \ar[r] & 0.
\end{tikzcd}
\]
\end{theorem}
\subsubsection{Construction of $\tensor*[^{\operatorname{tr}\!\! }]{\cA}{^{-1}}(\cE)$}\label{BS-construction}
Let $\Delta \cong \cX \subset \cX \times_\cM \cX$ denote the diagonal, and $p_1$ and $p_2$ the two projections of $\cX \times_\cM \cX$ to $\cX$.  For each of the projections $p_1, p_2$ we have a residue map $\operatorname{Res}^1, \operatorname{Res}^2$ along the fibres (cfr \cite{tate:1968,beilinson:1980,braunling:2018}).  The following is a key ingredient for us:
\begin{lemma}[{\cite[\S 2.1.1.1]{beilinson.schechtman:1988}}]
There exists a map $$\widetilde{\operatorname{Res}}: K_{\cX/\cM}\boxtimes K_{\cX/\cM}(3\Delta)  \rightarrow \cO_{\cX},$$ which vanishes on $K_{\cX/\cM}\boxtimes K_{\cX/\cM}(\Delta)$, is symmetric with respect to transposition, and such that $d\widetilde{\operatorname{Res}}=\operatorname{Res}^1-\operatorname{Res}^2$.  The restriction of $\widetilde{\Res}$ to $K_{\cX/\cM}\boxtimes K_{\cX/\cM}(2\Delta)$ gives a short exact sequence
\begin{equation*}\begin{tikzcd}[row sep=small]
0 \ar[r] & K_{\cX/\cM}\boxtimes K_{\cX/\cM}(\Delta) \ar[r] &  K_{\cX/\cM}\boxtimes K_{\cX/\cM}(2\Delta) \ar[dl, shorten=-2ex, "{\operatorname{res}_\Delta = \widetilde{\operatorname{Res}}}" description]
 \\   &K_{\cX/\cM}\boxtimes K_{\cX/\cM}(2\Delta)_{|\Delta}\cong \cO_{\cX} \ar[r]   & 0, \\
\end{tikzcd}
\end{equation*}
where the second map is $\widetilde{\Res}$, and coincides with the restriction to the diagonal $\Delta$.
\end{lemma}
We shall also need a particular description of the sheaf of (relative) first order differential operators $\diffone_{\cX/\cM}(\cE)$ (see \cite[2.1.1.2]{beilinson.schechtman:1988} or the introduction of \cite{esnault.tsai:2000}, from which we borrow the notation). 
Here and in what follows, we identify sheaves supported on the diagonal $\Delta$ with sheaves on $\cX$. The next
lemma is easily deduced from the definition of the ``pole at $\Delta$" map.
\begin{lemma}
The symbol short exact sequence for first order differential operators on $\cE$ relative to $f$ is isomorphic to the exact 
sequence
\begin{equation}\label{diagonals}
\begin{tikzcd}[row sep=small]
  0 \arrow[r] & \frac {\cE \boxtimes \cE^\circ (\Delta)}{\cE \boxtimes \cE^\circ} \arrow[r] \arrow[d, "\cong"] & \frac {\cE \boxtimes \cE^\circ (2\Delta)}{\cE \boxtimes \cE^\circ} \arrow[r] \arrow[d, "\delta"] & \frac {\cE \boxtimes \cE^\circ (2\Delta)}{\cE \boxtimes \cE^\circ (\Delta)} \arrow[r] \arrow[d, "\cong"] & 0 \\
  0 \arrow[r] & \cE nd(\cE) \arrow[r] & \diffone_{\cX/\cM}(\cE) \arrow[r,"\sigma_1"] & T_{\cX/\cM}\otimes \cE nd(\cE) \arrow[r] & 0.
 \end{tikzcd}
\end{equation}
where $\delta$ is the ``pole at $\Delta$'' map defined by 
$$\delta(\psi)(e) = \operatorname{Res}^2(\langle \psi, p_2^*(e) \rangle),$$
for any local section $\psi$ of $\frac{\cE \boxtimes \cE^\circ (2\Delta)}{\cE \boxtimes \cE^\circ}$ and any local section $e$ of $\cE$.
Here $\langle -,- \rangle$ is the natural pairing $\cE^\circ \times \cE \to  K_{\cX/\cM}$. 
\end{lemma}
We consider now the natural exact sequence
\begin{equation}\label{diagonals2}
\begin{tikzcd}[row sep=tiny] 
0 \ar[r] & \frac{\cE \boxtimes \cE^\circ}{\cE \boxtimes \cE^\circ (-\Delta)} \ar[d, equal] \ar[r] & \frac{\cE \boxtimes \cE^\circ(2\Delta)}
{\cE\boxtimes \cE^\circ (-\Delta)} \ar[r] & \frac{\cE \boxtimes \cE^\circ(2\Delta)}{\cE \boxtimes \cE^\circ} \ar[r] \ar[d, equal] & 0\\
& \cE nd(\cE) \otimes K_{\cX/\cM} & & \diffone_{\cX/\cM}(\cE). & \\
\end{tikzcd}
\end{equation}
Then the construction that defines the short exact sequence (\ref{at1}) is obtained by taking first the pull-back of (\ref{diagonals2}) to $\atalg_{\cX/\cM}(\cE)\subset \diffone_{\cX/\cM}(\cE)$, and then the push-out 
under the trace map $\cE nd(\cE)\otimes K_{\cX/\cM} \stackrel{\operatorname{Tr}}{\to} K_{\cX/\cM}$,
\begin{equation}\label{diag_A-1}
  \begin{tikzcd}[row sep=small]
    0 \arrow[r] & \frac{\cE \boxtimes \cE^\circ}{\cE \boxtimes 
    \cE^\circ (-\Delta)} \arrow[r] & \frac{\cE \boxtimes \cE^\circ(2\Delta)}{\cE \boxtimes \cE^\circ (-\Delta)} \arrow[r]  & \frac{\cE \boxtimes \cE^\circ(2\Delta)}{\cE \boxtimes \cE^\circ} \arrow[r] & 0 \\
    0 \arrow[r]& \cE nd(\cE) \otimes K_{\cX/\cM} \arrow[r] \arrow[d, "\operatorname{Tr}"]\ar[u, equal] & \tensor*[^{\operatorname{tr}\!\!}]{\widetilde{\cA}}{^{-1}}(\cE) \arrow[r] \arrow[d] \ar[u]& \atalg_{\cX/\cM}(\cE) \arrow[r] \arrow[d, equal]  \ar[u]& 0 \\
    0 \arrow[r] & K_{\cX/\cM} \arrow[r] & \tensor*[^{\operatorname{tr}\!\!}]{\cA}{^{-1}}(\cE) \arrow[r] & \cA_{\cX/\cM}(\cE) \arrow[r] & 0.
  \end{tikzcd}
\end{equation}
\subsection{The quasi-isomorphic Bloch--Esnault complex $\cB^\bullet$}\label{sectiononBE}
Following \cite{bloch.esnault:2002}, we will now construct a subcomplex $\cB^\bullet(\cE) \subset \tensor*[^{\operatorname{tr}\!\! }]{\cA}{^\bullet}(\cE)$ that allows for more handy computations. 
Its construction relies on 
the existence of a splitting of the short exact sequence
\begin{equation}\label{TfS_ses}
  \begin{tikzcd}
    0 \arrow[r] & T_{\cX/\cM} \arrow[r] & T_{f/S} \arrow[r, "df"] & f^{-1} T_{\cM/S} \arrow[r] \arrow[l, dashed, bend left=30] & 0.
  \end{tikzcd}
\end{equation}
\begin{remark}\label{fibredproduct}Note that this condition is in particular satisfied whenever $\cX$ is a fibered product $\cX = \cY \times_{S} \cM$ and $f = \pi_2$ the projection, since then $T_{\cX/S} \cong \pi_1^\ast T_{\cY/S} \oplus \pi_2^\ast T_{\cM/S}$ and in particular
\[
   T_{f/S} \cong \pi_1^\ast T_{\cY/S} \oplus f^{-1} T_{\cM/S} .
\]\end{remark}
\subsubsection{Construction of $\cB^\bullet(\cE)$} The definition of $\cB^{-1}(\cE)$ is analogous to that of $\tensor*[^{\operatorname{tr}\!\! }]{\cA}{^{-1}}(\cE)$ via the sub-quotient (\ref{diag_A-1}). One starts once again from the short exact sequence (\ref{diagonals2}), but pulls it back all the way to $\cE nd(\cE) \hookrightarrow \diffone_{\cX/\cM}(\cE)$, and then pushes out along the trace
\begin{equation}\label{diag_B-1}
  \begin{tikzcd}[row sep=small]
    0 \arrow[r] & \frac{\cE\boxtimes \cE^\circ}{\cE\boxtimes \cE^\circ (-\Delta)} \arrow[r] & \frac{\cE\boxtimes \cE^\circ(2\Delta)}{\cE\boxtimes \cE^\circ (-\Delta)} \arrow[r]  & \frac{\cE\boxtimes \cE^\circ(2\Delta)}{\cE\boxtimes \cE^\circ} \arrow[r] & 0 \\
    0 \arrow[r]& \cE nd(\cE) \otimes K_{\cX/\cM} \arrow[r] \arrow[d, "\operatorname{Tr}"]\ar[u, equal] & \widetilde{\cB}^{-1} \arrow[r] \arrow[d] \ar[u]& \cE nd(\cE) \arrow[r] \arrow[d, equal]  \ar[u]& 0 \\
    0 \arrow[r] & K_{\cX/\cM} \arrow[r] & \cB^{-1} \arrow[r] & \cE nd(\cE) \arrow[r] & 0.
  \end{tikzcd}
\end{equation}
Similarly, we define $\cB^0(\cE)$ via the pull-back of the symbol exact sequence of $\tensor*[^{\operatorname{tr}\!\! }]{\cA}{^0}(\cE) = \cA_{f/S}(\cE)$ under the inclusion $f^{-1} T_{\cM/S} \hookrightarrow T_{f/S}$ arising through the splitting condition on (\ref{TfS_ses}), so that we have the following diagram
\begin{equation*}
  \begin{tikzcd}[row sep=small]
   0 \ar[r] & \cE nd(\cE) \ar[d, equals] \ar[r] & \cB^0(\cE) \ar[d] \ar[r] & f^{-1} T_{\cM/S} \ar[d] \ar[r] & 0\\
   0 \ar[r] & \cE nd(\cE) \ar[r] & \tensor*[^{\operatorname{tr}\!\! }]{\cA}{^0} = \cA_{f/S}(\cE) \ar[r] & T_{f/S} \ar[r] & 0.
 \end{tikzcd}
\end{equation*}
Hence $\cB^\bullet(\cE)$ is a subcomplex of $\tensor*[^{\operatorname{tr}\!\! }]{\cA}{^{\bullet}}(\cE)$, and the following holds true.
\begin{proposition}[{\cite[Sect. 5.2]{bloch.esnault:2002}}]
\label{quasiiso}
If the short exact sequence (\ref{TfS_ses}) is split, the complex $\cB^\bullet(\cE)$ is quasi-isomorphic to $\tensor*[^{\operatorname{tr}\!\! }]{\cA}{^{\bullet}}(\cE)$.
\end{proposition}
\begin{corollary}
The short exact sequence of complexes (\ref{3cpxes}) is quasi-isomorphic to
\[
\begin{tikzcd}[row sep=small]
& \mathcal{O}_{\cX} \ar[r, equals] \ar[d, "d_{\cX/\cM}"]& \cB^{-2}(\cE) \ar[d]& & \\ 0\ar[r] & K_{\cX/\cM} 
\ar[r]&  \cB^{-1}(\cE) \ar[r] \ar[d]& \End(\cE) \ar[r] \ar[d] & 0 \\
& & \cB^{0}(\cE) \ar[r, equals] & \cB^{0}(\cE) . &
\end{tikzcd}
\]
\end{corollary}
Moreover, since we are considering only $0^{th}$ direct images, we can drop the degree $-2$ part of the first two complexes. Hence we obtain a short exact sequence of complexes,
$$0 \to K_{\cX/\cM}[1] \to \cB^\bullet(\cE) \to \mathcal{C}^\bullet(\cE) \to 0,$$
where $\mathcal{C}^{-1}(\cE) := \cE nd(\cE)$ and $\mathcal{C}^0(\cE):= \cB^0(\cE)$. We also observe that $\mathcal{C}^\bullet(\cE)$ is quasi-isomorphic to $f^{-1}T_{\cM/S}$ since this is exactly the cokernel of $\cE nd (\cE) \to \cB^0(\cE)$. Thus Theorem \ref{main1} now simplifies to
\begin{theorem}\label{easyBS}
We have an isomorphism of short exact sequences
\begin{equation*}
 \begin{tikzcd}[column sep=3ex, row sep=small]
   0 \ar[r] & R^0f_*(K_{\cX/\cM}[1]) \ar[d, "\cong"] \ar[r] & R^0f_*(\cB^\bullet(\cE)) \ar[d, "\cong"] \ar[r] & R^0f_*(\cE nd(\cE) \to \cB^0(\cE)) \cong T_{\cM/S} \ar[d, "\cong"] \ar[r] &0\\
0 \ar[r] & \cO_\cM \ar[r] & \cA_{\cM/S}(\lambda(\cE)) \ar[r] & T_{\cM/S} \ar[r] & 0.
 \end{tikzcd}
\end{equation*}
\end{theorem}
\begin{remark}
We observe that both sides of the central vertical isomorphism depend on $\cE$.
\end{remark}
\subsubsection{Traceless version $\tensor*[^0]{\cB}{^\bullet}(\cE)$ of $\cB^\bullet(\cE)$}\label{rmk_tracelessB} 
As expected, we define the subsheaf $\tensor*[^0]{\cB}{^{-1}}(\cE)\subset\cB^{-1}(\cE)$ via the pull-back of the short exact sequence defining $\cB^{-1}(\cE)$ in (\ref{diag_B-1}) along the inclusion of traceless endomorphisms $\cE nd^0(\cE) \hookrightarrow \cE nd(\cE)$,
\begin{equation*}
 \begin{tikzcd}[row sep=small]
   0 \ar[r] & K_{\cX/\cM} \ar[d, equals] \ar[r] & \tensor*[^0]{\cB}{^{-1}}(\cE) \ar[d] \ar[r] & \cE nd^0(\cE) \ar[d] \ar[r] & 0\\
   0 \ar[r] & K_{\cX/\cM} \ar[r] & \cB^{-1}(\cE) \ar[r] & \cE nd(\cE) \ar[r] & 0.
 \end{tikzcd}
\end{equation*}
As we did before, we introduce also a quotient sheaf $\tensor*[^0]{\cB}{^0}(\cE)$ of $\cB^0(\cE)$, obtained as push-out through $\cE nd(\cE) \rightarrow \cE nd^0(\cE)$, that is 
\begin{equation*}
 \begin{tikzcd}[row sep=small]
   0 \ar[r] & \cE nd(\cE) \ar[d] \ar[r] & {\cB}^0(\cE) \ar[d] \ar[r] & f^{-1}T_{\cM/S} \ar[d, equals] \ar[r] & 0\\
   0 \ar[r] & \cE nd^0(\cE) \ar[r] & \tensor*[^0]{\cB}{^0}(\cE) \ar[r] & f^{-1}T_{\cM/S} \ar[r] & 0.
 \end{tikzcd}
\end{equation*}
\subsubsection{Identification of $\cB^{-1}(\cE)$ and $\tensor*[^{0}]{\cB}{^{-1}}(\cE)$} The duality 
$$\cA_{\cX/\cM}(\cE)^\ast \cong \cB^{-1}(\cE)$$
was already stated in \cite{bloch.esnault:2002}
formula (5.31). We give a proof here, in particular to include a discussion of the traceless case, and to control the necessary restrictions on the characteristic of the ground field.
\begin{theorem}\label{thmdualityB}
There is a canonical identification between the natural short exact sequences
\[
 \begin{tikzcd}[row sep=small]
   0 \arrow[r] & T_{\cX/\cM}^\ast \arrow[r] \arrow[d, equal] & \cA_{\cX/\cM}(\cE)^\ast \arrow[r] \arrow[d, "\cong"] & \cE nd(\cE)^\ast \arrow[r] \arrow[d, "\cong", "- \operatorname{Tr}" swap] & 0 \\
   0 \arrow[r] & K_{\cX/\cM} \arrow[r] & \cB^{-1}(\cE) \arrow[r] & 
   \cE nd(\cE) \arrow[r] & 0
 \end{tikzcd}
\]
There is also a traceless analogue:
\[
 \begin{tikzcd}[row sep=small]
   0 \arrow[r] & T_{\cX/\cM}^\ast \arrow[r] \arrow[d, equal] & \cA^0_{\cX/\cM}(\cE)^\ast \arrow[r] \arrow[d, "\cong"] & \cE nd^0(\cE)^\ast \arrow[r] \arrow[d, "\cong", "- \operatorname{Tr}" swap] & 0 \\
   0 \arrow[r] & K_{\cX/\cM} \arrow[r] & \tensor*[^{0}]{\cB}{^{-1}}
   (\cE) \arrow[r] & \cE nd^0(\cE) \arrow[r] & 0 .
 \end{tikzcd}
\]
\end{theorem}
\begin{remark}
Note that the vertical maps on the RHS are given by the opposite of the
isomorphism induced by the trace pairing.
\end{remark}
\begin{proof}
Following \cite[Sect. 2.1.1.3]{beilinson.schechtman:1988}, let us define a pairing
\begin{eqnarray*}
  \cE \boxtimes \cE^\circ (2\Delta) \times \cE \boxtimes \cE^\circ (\Delta) & \to & \cO_{\cX};\\
  (\psi_1,\psi_2) & \mapsto & \widetilde{\Res}(\psi_1\cdot^t\psi_2);
\end{eqnarray*}
where $\prescript{t}{}{\psi_2}$ denotes the transposition of $\psi_2$, that is the pull-back under the map that exchanges the two factors of the fibered product $\cX \times_{\cM} \cX$. This means that $\prescript{t}{}{\psi_2}$ is a section of $\cE^\circ \boxtimes \cE(\Delta)$. Then we observe that the product $\psi_1\cdot^t\psi_2$ is a section of $K_{\cX/\cM}\boxtimes K_{\cX/\cM}(3\Delta)$,
after taking the trace $\operatorname{Tr} : \cE \otimes \cE^\circ
\to K_{\cX/\cM}$ on each factor. Since $\widetilde{\Res}$ is zero on $K_{\cX/\cM}\boxtimes K_{\cX/\cM}(\Delta)$, the pairing descends to a pairing on the quotients
\[
  \langle - , - \rangle : \frac{\cE\boxtimes \cE^\circ (2\Delta)}{\cE\boxtimes \cE^\circ} \times \frac{\cE \boxtimes 
  \cE^\circ (\Delta)}{\cE \boxtimes \cE^\circ (-\Delta)}  \to  
  \cO_{\cX} .
\]
We claim that this pairing is non-degenerate. In order to check this, observe that it is defined on the central terms of the two short exact sequences (\ref{diagonals}) and (\ref{diagonals2}),
\[
   \begin{tikzcd}[row sep=small, column sep=small]
    \cE nd(\cE) \cong \frac{\cE\boxtimes \cE^\circ (\Delta)}
    {\cE\boxtimes \cE^\circ} \ar[d, hook] & \frac{\cE\boxtimes \cE^\circ}{\cE\boxtimes \cE^\circ(-\Delta)}\cong \cE nd(\cE) 
    \otimes K_{\cX/\cM} \ar[d, hook] & \\
    \diffone_{\cX/\cM}(\cE) \cong \frac{\cE \boxtimes \cE^\circ (2\Delta)}{\cE \boxtimes \cE^\circ} \ar[d, swap, "\sigma_1", two heads] \arrow[r, phantom,  "\times"] & \frac{\cE \boxtimes \cE^\circ (\Delta)}{\cE \boxtimes \cE^\circ(-\Delta)} \ar[r, "{\langle - , - \rangle}"] \ar[d, two heads] & \cO_{\cX} \\
    \cE nd(\cE)\otimes T_{\cX/\cM} \cong \frac{\cE \boxtimes 
    \cE^\circ (2\Delta)}{\cE \boxtimes \cE^\circ (\Delta)} & \frac{\cE \boxtimes \cE^\circ (\Delta)}{\cE \boxtimes \cE^\circ}\cong \cE nd(\cE)  &
  \end{tikzcd}
\]
Using the fact that $\widetilde{\Res}$ vanishes on $K_{\cX/\cM}\boxtimes K_{\cX/\cM}(\Delta)$, we note that
the pairing is identically zero when restricted to the product of the kernels $\frac{\cE \boxtimes \cE^\circ (\Delta)}{\cE\boxtimes \cE^\circ} \times \frac{\cE \boxtimes \cE^\circ}{\cE \boxtimes \cE^\circ(-\Delta)}$. Therefore it induces pairings on the products of the kernel of one sequence with the quotient of the other one, that is, on $\cE nd(\cE) \times \cE nd(\cE)$ and $\cE nd(\cE) \otimes T_{\cX/\cM} \times \cE nd(\cE)\otimes K_{\cX/\cM}$. 
\begin{lemma}
The residue pairing $\langle - , - \rangle$ factorizes through the trace pairings
$- \operatorname{Tr}$ on $\cE nd(\cE) \times \cE nd(\cE)$ and
$+ \operatorname{Tr}$  on $\cE nd(\cE) \otimes T_{\cX/\cM} \times \cE nd(\cE)\otimes K_{\cX/\cM}$.
\end{lemma}
\begin{proof}
Consider $\psi_1$ a local section of $\frac{\cE\boxtimes \cE^\circ (\Delta)} {\cE\boxtimes \cE^\circ} \subset 
\frac{\cE \boxtimes \cE^\circ (2\Delta)}{\cE \boxtimes \cE^\circ}$ and
$\psi_2$ a local section of $\frac{\cE\boxtimes \cE^\circ (\Delta)} {\cE\boxtimes \cE^\circ(-\Delta)}$. As explained above 
$\langle \psi_1, \psi_2 \rangle$ depends only on
$\langle \psi_1, \overline{\psi_2} \rangle$, where $\overline{\psi_2}$
is the class of $\psi_2$ in $\frac{\cE\boxtimes \cE^\circ (\Delta)} {\cE\boxtimes \cE^\circ}$. It will be enough to do the computations
locally. Choose (as in \cite{esnault.tsai:2000}) a local coordinate
$x$ at a point $p \in \cX$ and let $(x,y)$ be the induced local
coordinate at the point $(p,p) \in \Delta$. Then the local equation 
of $\Delta$ is $x-y = 0$. Let $e_i$ be a local basis of $\cE$ and
$e_j^*$ its dual basis. Then we can write the local sections
$\psi_1$ and $\overline{\psi_2}$ as
$$ \psi_1 = \sum_{i,j} e_i \otimes e_j^* \frac{\alpha_{ij}(x,y-x)}
{y-x}dy \ \  \text{and} \ \  \overline{\psi_2} = \sum_{k,l} 
e_k \otimes e_l^* \frac{\beta_{kl}(x,y-x)}
{y-x}dy $$
for some local regular functions $\alpha_{ij}$ and $\beta_{kl}$.
Then the local sections $\phi_1$ and $\phi_2$ of $\cE nd(\cE)$ associated
to $\psi_1$ and $\overline{\psi_2}$ are given by
$$ \phi_1 = \sum_{i,j}  e_i \otimes e_j^* \alpha_{ij}(x,0) \ \ 
\text{and} \ \ \phi_2 = \sum_{k,l}  e_k \otimes e_l^* \beta_{kl}(x,0).
$$
Then we compute
\begin{eqnarray*}
\langle \psi_1, \overline{\psi_2} \rangle & = & 
\widetilde{\Res}\left( \sum_{ijkl} e_i \otimes e_l^* \cdot e_k \otimes e_j^* \frac{\alpha_{ij}(x,y-x) \beta_{kl}(y,x-y)}{-(x-y)^2}
dxdy\right) \\
 & = & \widetilde{\Res}\left( \sum_{ij} \frac{\alpha_{ij}(x,y-x) \beta_{ji}(y,x-y)}{-(x-y)^2} dxdy \right) \\
 & = & -  \sum_{ij} \alpha_{ij}(x,0) \beta_{ji}(x,0) = -
 \operatorname{Tr}(\phi_1 \phi_2).
\end{eqnarray*}
The computations for the second case are similar.
\end{proof}
Since the trace pairing $ \operatorname{Tr}$ is 
non-degenerate, we deduce from the
above Lemma that the pairing $\langle - , - \rangle$ is also non-degenerate.

Now, we observe that $\cA_{\cX/\cM}(\cE) \subset 
\frac{\cE \boxtimes \cE^\circ(2\Delta)}{\cE \boxtimes \cE^\circ}$ 
and that $\frac{\cE \boxtimes \cE^\circ(\Delta)}{\cE \boxtimes \cE^\circ(-\Delta)}\twoheadrightarrow \cB^{-1}(\cE)$. We want to prove that the restriction $\langle \cA_{\cX/\cM}(\cE), - \rangle$ descends to $\cB^{-1}(\cE)$, but this follows from the definition of $\cA_{\cX/\cM}(\cE)$ by pull-back via $T_{\cX/\cM}\otimes \cE nd(\cE)$ and the definition of $\cB^{-1}(\cE)$ by push-out 
via $\cE nd(\cE)\otimes K_{\cX/\cM} \stackrel{Tr}{\twoheadrightarrow} K_{\cX/\cM}$, and the duality between these two maps. Hence we obtain a non-degenerate pairing
\[
\langle - , - \rangle: \cA_{\cX/\cM}(\cE) \times \cB^{-1}(\cE) \longrightarrow \cO_{\cX}.
\]
The same argument yields non-degeneracy of the traceless version of this pairing
\[
\langle - , - \rangle: \cA_{\cX/\cM}^0(\cE) \times \tensor*[^{0}]{\cB}{^{-1}}(\cE) \longrightarrow \cO_{\cX}.
\]
\end{proof}
\begin{remark}
The duality between $\cA_{\cX/\cM}(\cE)$ and $\cB^{-1}(\cE)$ was constructed by Sun-Tsai in \cite[Lemma 4.11.2]{sun.tsai:2004} using a local description of $\cB^{-1}(\cE)$. Note that their claim involves the Atiyah algebroid $\cA_{\cX/\cM}(\cE^*)$, which is isomorphic to $\cA_{\cX/\cM}(\cE)$ but has opposite extension class. 
\end{remark}

\begin{remark}
We note that

$$\frac{\cE\boxtimes \cE^\circ(\Delta)}{\cE\boxtimes \cE^\circ(-\Delta)}\cong \mathcal{D}^{(1)}_{\cX/\mathcal{M}}(\cE) \otimes K_{\cX/\mathcal{M}}.$$

Thus the pairing $\langle -,- \rangle $ described in the above proof induces a natural isomorphism between $\mathcal{D}^{(1)}_{\cX/\mathcal{M}}(\cE)^*$ and $\mathcal{D}^{(1)}_{\cX/\mathcal{M}}(\cE) \otimes K_{\cX/\mathcal{M}}$.
\end{remark}

\section{The splitting of the adjoint map.}\label{appendixsplitting}
In this appendix we collect some representation-theoretical facts needed in the proof of Prop. \ref{isodirimage}. We will work in the following framework. We will denote by $\cE$ a rank $r$ vector bundle on a smooth algebraic variety $X$ and as usual $\mathcal{E}nd^0(\cE)$ will denote the traceless endomorphisms of $\cE$. We need the characteristic $p$ of the field $\Bbbk$ to be $0$ or not dividing $r$.

First we observe that we have two non-degenerate pairings induced by the trace,
\begin{eqnarray}
\operatorname{Tr} : \cE nd(\cE)\times \cE nd(\cE) & \to & \cO_X , \label{trone}\\
\operatorname{Tr} : \cE nd(\cE nd(\cE)) \times \cE nd(\cE nd(\cE)) & \to & \cO_X , \label{trtwo}
\end{eqnarray}
which allow us to identify $\cE nd (\cE)$ with $\cE nd(\cE)^*$ and $\cE nd(\cE nd(\cE))$ with 
$\cE nd(\cE nd(\cE))^*$. Moreover, we denote by
\begin{eqnarray}
\ad: \cE nd(\cE) & \to & \cE nd(\cE nd(\cE)) \label{defad}\\
\alpha & \mapsto & (\beta \mapsto [\alpha, \beta])  \nonumber
\end{eqnarray}
the $\cO_X$-linear map given by the adjoint, for any local sections $\alpha, \beta$ of $\cE nd(\cE)$.
\begin{lemma} \label{splitting}
Let $\alpha,\beta$ be local sections of the vector bundle $\cE nd(\cE)$. The $\cO_X$-linear map
\begin{eqnarray*}
s:\cE nd(\cE nd(\cE)) \cong \cE nd (\cE)\otimes \cE nd(\cE) & \to & 
\cE nd(\cE)\\
\alpha \otimes \beta & \mapsto & \frac{1}{2r}[\beta,\alpha]
\end{eqnarray*}
satisfies $s \circ \ad(\alpha) = \alpha - \frac{\tr(\alpha)}{r}\Id_{\cE}$, \it i.e. \rm $s$ is a splitting of the restriction of $\ad$ to $\cE nd^0(\cE)$.
\end{lemma}
\begin{proof}
It will be enough to check the equality pointwise. The statement then reduces to check that for an $r \times r$ matrix $A \in 
\mathrm{M}_r(\Bbbk)$ we have the equality $s \circ \ad (A) =
A - \frac{\tr(A)}{r} I_r$. We consider the canonical basis 
$\{ E_{ij}  \}$ with $1 \leq i,j \leq r$ of $\mathrm{M}_r(\Bbbk)$.
The dual basis of $\{ E_{ij}  \}$ under the trace pairing (\ref{trone})
is given by $\{ E_{ji} \}$. The claim then follows by straightforward 
computation :
\begin{eqnarray*}
s \circ \ad(A) & = & s \left( \sum_{i,j} E_{ji} \otimes [A,E_{ij}]    \right)  = \frac{1}{2r}  \sum_{i,j} AE_{ij}E_{ji} - E_{ij}AE_{ji} -
 E_{ji}AE_{ij} + E_{ji}E_{ij}A \\
 & = & \frac{1}{2r} (2r A - 2\tr(A) I_r). 
\end{eqnarray*}
\end{proof}
\begin{lemma} \label{dualofsplitting}
Using the identifications (\ref{trone}) and  (\ref{trtwo}) given by the trace pairings we denote by $s^* :  \cE nd(\cE) \to  \cE nd(\cE nd(\cE))$
the dual of $s$. Then we have the equality
$$ s^* = \frac{1}{2r} \ad.$$
\end{lemma}
\begin{proof}
As in the previous lemma we will check the equality pointwise. By the 
definition of the dual map $s^*$ and the trace pairings (\ref{trone}) and
(\ref{trtwo}) it is easily seen that the claimed equality is equivalent to
the equality
$$ \tr (\ad(A). B \otimes C) = \tr (A [C,B])$$
for any matrices $A,B,C \in \mathrm{M}_r(\Bbbk)$. Note that the trace
on the left-hand side is the trace on $End(\mathrm{M}_r(\Bbbk)) \cong \mathrm{M}_r(\Bbbk) \otimes \mathrm{M}_r(\Bbbk)$. Again this equality
is proved by straightforward computation :
\begin{eqnarray*}
\tr(\ad(A). B \otimes C) & = & \sum_{i,j} \tr ( 
E_{ji} \otimes [A,E_{ij}] \otimes B \otimes C ) = \sum_{i,j} \tr(E_{ji} C ) \tr( [A,E_{ij}] B) \\
& = & \sum_{i,j} (\tr(E_{ji} C ) ( \tr (BAE_{ij}) - \tr (E_{ij}AB) )  =  \tr(BAC) - \tr(ABC) \\
& = &  \tr(A[C,B]).
\end{eqnarray*}
\end{proof}
We will also abuse slightly of notation and denote also by $\ad$ the $\cO_X$-linear map 
$\cE nd(\cE) \to  \cE nd(\cE nd^0(\cE))$ induced by the one defined in (\ref{defad}). 
We will write instead $\ad_0:\cE nd^0(\cE) \to  \cE nd^0(\cE nd^0(\cE))$ for the restriction 
to $\cE nd^0(\cE)$.
\begin{proposition}\label{ST310}
\begin{enumerate}[(a)]
\item There exists a $\cO_X$-linear map
$$
\widetilde{\ad}: \cA(\cE) \to \cA(\cE nd^0(\cE)),
$$
extending respectively $\ad$ inducing the identity on $T_X$. Note that $\widetilde{\ad}$ factorizes
through $\cA^0(\cE)$. We shall denote by 
$$
\widetilde{\ad}_0 : \cA^0(\cE) \to \cA(\cE nd^0(\cE))
$$
the factorized map.
\item There exists a $\cO_X$-linear map
$$
\widetilde{s}: \cA(\cE nd^0(\cE)) \to \cA^0(\cE),
$$
extending $s: \cE nd(\cE nd^0(\cE)) \to \cE nd^0(\cE)$, inducing
the identity on $T_X$ and such that $\widetilde{s} \circ \widetilde{\ad}_0 = Id_{\cA^0(\cE)}$.
\item With the notation of Appendix A, there exists a $\cO_\cX$-linear map
$$\widehat{\ad}:{}^0\cB^{-1}(\cE) \to{}^0\cB^{-1}(\cE nd^0(\cE)),$$
lifting $\ad_0$ and inducing $2r \Id$ on the line subbundle $K_{\cX/\cM}$.
\end{enumerate}
\end{proposition}
\begin{proof}
Part (a) is proved in \cite{atiyah:1957} pages 188-189. \\
Part (b): We define $\widetilde{s}$ as the push-out of the
exact sequence
$$ 0 \to \cE nd^0 (\cE nd^0(\cE)) \to \cA^0(\cE nd^0(\cE)) \to
T_X \to 0 $$
under the $\cO_X$-linear map $s$. Then, by Lemma \ref{splitting}, since
$s$ is a splitting of $\ad_0$, we see that the extension class of the push-out 
is the same as the extension class of $\cA^0(\cE)$, hence these two vector bundles
are isomorphic (see e.g. \cite{atiyah:1957} pages 188-189). \\
Part (c):  We recall from Theorem \ref{thmdualityB} that there exist
isomorphisms 
$$ \delta_{\cE} : \cA^0_{\cX/\cM}(\cE)^* \to {}^0\cB^{-1}(\cE) \ \ \text{and} \ \ 
 \delta_{\cE  nd^0(\cE)} : \cA^0_{\cX/\cM}(\cE  nd^0(\cE))^* \to {}^0\cB^{-1}(\cE  nd^0(\cE)) $$
We then construct the map $\widehat{\ad}$ as the composition 
$$\widehat{\ad} = (2r) \delta_{\cE  nd^0(\cE)} \circ \widetilde{s}^* \circ \delta_{\cE}^{-1}.$$
Then $\widehat{\ad}$ induces $(2r) \Id$ on $K_{\cX/\cM}$ and, by Lemma \ref{dualofsplitting},
$\widehat{\ad}$ lifts the map $\ad_0$.
\end{proof}
\begin{remark}
Proposition \ref{ST310} coincides with \cite[Prop. 3.10]{sun.tsai:2004}. Our proof is different since we give a global construction of the liftings of the adjoint maps.
\end{remark}

\section{Basic facts about the moduli space $\cM$ through the Hitchin system}\label{appendixbasicfacts}
In this appendix we give proofs for some of the basic facts about the moduli space of stable  bundles $\cM$ (as in Section \ref{sect_basicfacts}) that we use in the main body of the paper.   These are essentially all well known, but we were unable to find references for them in the generality we need (outside the complex case).  We therefore show here how they can all be obtained using the Hitchin system -- a strategy once again due to Hitchin (cfr. \cite[\S 6]{hitchin:1987} and \cite[\S 5]{hitchin:1990}) -- via some minor adaptations to the algebro-geometric setting.
\subsection{The moduli space of Higgs bundles and the Hitchin system} We will denote by $\mathcal{M}^{\higgs, \semis}$ the moduli space of semi-stable Higgs bundles with trivial determinant (and trace-free Higgs field) -- all still relative over $S$ as before.  
This space is singular but normal, and comes equipped with the Hitchin system, a projective morphism $\phi$ to the vector bundle $\pi_{\HH}:\HH\rightarrow S$ associated to the sheaf $\oplus_{i=2}^r \pi_{s *} K_{\calC/S}^i$ over $S$.  This morphism is equivariant with respect to the $\mathbb{G}_m$-action that scales the Higgs fields, and acts with weight $i$ on $\pi_{s *} K^i_{\calC/S}$.
  The fibers of $\pi_{\higgs}:\mathcal{M}^{\higgs, \semis}\rightarrow S$ have a canonical (algebraic) symplectic structure on their smooth locus, which extends the one on $T^*_{\cM/S}$.  Closed points in $\HH$ give rise to degree $r$ spectral covers of $\calC$.  The locus whose spectral curve is smooth is denoted by $\HH^{\operatorname{reg}}$.
\subsection{Proofs}
\begin{proposition}[{Proposition \ref{basicfacts}(\ref{basicfactsthree})}]
There are no global vector fields on $\cM$: $$\pi_{e*}T_{\cM/S}=\{0\}.$$
\end{proposition}
\begin{proof} Elements of $\pi_{e*}T_{\cM/S}$ would give rise to global functions on $T^*_{\cM/S}$.  As the complement of $\cM$ in $\mathcal{M}^{\higgs, \semis}$ has sufficiently high codimension, these would extend by Hartogs's theorem to all of $\mathcal{M}^{\higgs, \semis}$.  As they have weight $1$ under the $\mathbb{G}_m$-action, they have to be pulled-back from functions on $\HH$ of the same weight, but there are no such functions.
\end{proof}
\begin{proposition}\label{rho-Hit-isom}
The Hitchin symbol $\rho^{\Hit}$ is an isomorphism.
\end{proposition}
\begin{proof}
Elements of ${\pi_e}_\ast \Sym^2 T_{\cM / S}$ can be understood as regular functions on the total space of $T^*_{\cM/S}$, of degree $2$ on all tangent spaces.  In turn these extend, by Hartog's theorem, to $\cM^{\higgs,\semis}$, where they are of degree 2 with respect to the $\mathbb{G}_m$-action that scales the Higgs field.  As the Hitchin system is equivariant, they are moreover obtained from regular linear functions on the quadratic part of the Hitchin base, which is exactly given by $R^1 {\pi_s}_\ast T_{\calC / S}$ though $\rho^{\Hit}$.
\end{proof}
To establish that $\mu_{\LL^k}$ is injective, we can again adapt the reasoning from \cite[\S 5]{hitchin:1990}.  By Propositions \ref{thm_mu_O} and \ref{phi-rho-L}, and Lemma \ref{rho-Hit-isom}, it suffices to show that $\Phi$ is injective.
   \begin{lemma}[{\cite[Proposition 5.2]{hitchin:1990}}]\label{useful-lemma}
There exists a canonical isomorphism $$\begin{tikzcd}\Psi:\pi_{\HH*}\mathcal{\cO_{\HH}}\otimes \HH^*\ar[r] & R^1\pi_{\higgs *}\cO,\end{tikzcd}$$  of $\pi_{\HH*}\mathcal{\cO_{\HH}}$-modules
which is equivariant with respect to the natural action of $\mathbb{G}_m$ on $\pi_{\HH*}\mathcal{\cO_{\HH}}\otimes \HH$, and the natural action twisted by weight $-1$ on $R^1\pi_{\higgs *}\cO$.
\end{lemma}
\begin{proof} Indeed, sections of $\HH^*$ give rise to fiber-wise linear functions on $\HH$, which pull back by $\phi$ to functions on $\cM^{\higgs}$.   As the latter has an algebraic symplectic structure on $\cM^{\higgs, \stable}$ extending the canonical one on $T^*{\cM}$, these give rise to hamiltonian vector fields on $\cM^{\higgs, \stable}$ which are tangent to the fibres of $\phi$.  Moreover, the inverse of the determinant-of-cohomology line bundle $\cL$ naturally extends to $\cM^{\higgs}$, and is relatively ample with respect to $\phi$.  Taking the cup product with its relative Atiyah class gives a natural morphism $\pi_{\higgs *} T_{\cM^{\higgs}/S}\rightarrow R^1\pi_{\higgs *}\cO$.  The composition gives a morphism $\HH\rightarrow  R^1\pi_{\higgs *}\cO$, which naturally extends as a morphism of $\pi_{\HH*}\mathcal{\cO_{\HH}}$-modules  to the desired morphism $\Psi$.

To show that $\Psi$ is an isomorphism, it can be argued as follows: as $\pi_{\higgs}$ factors over $\pi_{\HH}$, and the latter is an affine morphism, we have that $R^1\pi_{\higgs *}\cO_{\cM^{\higgs}}\cong \pi_{\HH*}\left(R^1\phi_{*}\cO_{\cM^{\higgs}}\right)$.  Now, through the theory of abelianisation, we know that over a locus $\HH^{\circ}$ whose complement has sufficiently high co-dimension, the morphism $\phi$ is a family of (semi-)abelian varieties.  The line bundle $\cL$ restricts to an ample one on the fibres, and for those fibres $X$ it is known that cupping with $[\cL]$ is an isomorphism $H^0(X, T_X)\rightarrow H^1(X, \cO_X)$.  As the vector fields on $\cM^{\higgs}$ are independent, on each such $X$ the space $H^0(X, T_X)$ is given by the vector field coming from $\HH^*$.  As a result, we find that, on $\HH^{\circ}$, $R^1\phi_{*}\cO_{\cM^{\higgs}}$ is a trivial vector bundle, and that the map $\Psi$ is indeed an isomorphism.

It is also straightforward to observe that the map $\Psi$ is in fact equivariant for the natural $\mathbb{G}_m$-action that is defined on all spaced, induced by the scaling of Higgs fields, provided that we twist the action on $R^1\pi_{\HH*}\cM^{\higgs}$ by a weight $-1$.  \end{proof}

\begin{proposition}[{\ref{basicfacts}(\ref{basicfactsfour})}] \label{nor1}We have that $R^1\pi_{e*}\cO_{\cM}=\{0\}$.
\end{proposition}
\begin{proof}
It suffices to remark that sections of $R^1\pi_{e*}\cO_{\cM}$ correspond to sections of\\ $R^1\pi_{\higgs*}\cO_{\cM^{\higgs,\semis}}$ of weight $0$, which would correspond under $\Psi$ to sections of weight $-1$, of which there are none.
\end{proof}
\begin{proposition}\label{cap-isom}
The map $\cup[\LL]:\pi_{e*}\Sym^2T_{\cM/S}\rightarrow R^1\pi_{e*} T_{\cM/S}$ is an isomorphism.
\end{proposition}
\begin{proof}
We now want to restrict the isomorphism $\Psi$ from \ref{useful-lemma} to the sub-bundle of $\pi_{\HH*}\cO_{\HH}\otimes \HH^*$ of weight 2, which corresponds exactly to fibre-wise linear functionals on $\pi_{s*}K^2_{\calC/S}$, which by relative Serre duality is exactly given by $R^1\pi_{s*} T_{\calC/S}$.  On this space $\Psi$ restricts to give an isomorphism to $R^1\pi_{e*}T_{\cM/S}$.  To show that this is a multiple of $\Phi$, one can argue as follows: if $\cO^{(1)}$ is the structure sheaf of the first order infinitesimal neighborhood of $\cM$ in $\cM^{\higgs}$ (cfr. \cite[\href{https://stacks.math.columbia.edu/tag/05YW}{Tag 05YW}]{stacks-project}), we have the short exact sequence on $\cM$ 
$$\begin{tikzcd} 0\ar[r] & N^*_{\cM/\cM^{\higgs}} \ar[r] & \cO^{(1)} \ar[r] & \cO_{\cM}\ar[r] & 0.
\end{tikzcd}
$$
Here $N^*_{\cM/\cM^{\higgs}}$ is the co-normal bundle of $\cM$ in $\cM^{\higgs}$, which is canonically isomorphic to the tangent bundle $T_{\cM/S}$.  As by  Proposition \ref{nor1} we have that $R^1\pi_{e*}\cO_{\cM}=\{0\}$, this gives $$R^1\pi_{e*} T_{\cM/S}\cong R^1\pi_{e*}\cO^{(1)}.$$
If $\mathcal{I}$ is the ideal sheaf of $\cM$ in $\cM^{\higgs}$, we have that $\cO^{(1)}=\left(\cO_{\cM^{\higgs}}\big/ \mathcal{I}^2\right) \Big|_{\cM}$, and hence we have a restriction map $$\begin{tikzcd}R^1\pi_{\higgs *}\cO_{\cM^{\higgs}}\ar[r] & R^1\pi_{e*}\cO^{(1)}\cong R^1\pi_{e*} T_{\cM/S}\end{tikzcd},$$
which is the identity on $R^1\pi_{e*} T_{\cM/S}$ (sitting inside $R^1\pi_{\higgs *}\cO_{\cM^{\higgs}}$ as the weight $1$ part).  So we only need to keep track of first order information in the normal direction.  We now claim that, for any $\Delta\in R^1\pi_{\higgs *}(\Omega^1_{\cM^{\higgs}/S})$ which restricts to $\widetilde{\Delta}\in R^1\pi_{e*}(\Omega^1_{\cM/S})$ the following diagram is commutative:
\begin{equation}\label{lastdiagram}\begin{tikzcd}[row sep=small,column sep=small]
\ &\pi_{e*}\Sym^2 T_{\cM/S} \ar[rr, "{\cup 2\widetilde{\Delta}}"] \ar[dl, hook]&\ & R^1\pi_{e*}T_{\cM / S} &\  \\
\pi_{\higgs*}\cO_{\cM^{\higgs}}\ar[dr, "d" '] &\ &\ &\ & R^1\pi_{e*}\cO^{(1)}\ar[ul, "{\cong}" ']\\
\ & \pi_{\higgs *}\Omega^1_{\cM^{\higgs}/S}\ar[r, "\omega", "\cong"'] & \pi_{\higgs*}T_{\cM^{\higgs}/S}\ar[r, "{\cup\Delta}"]
& R^1\pi_{\higgs*}\cO_{\cM^{\higgs}}\ar[ur, "{\text{restrict}}" '] &\
\end{tikzcd}\end{equation}
In \cite[page 379]{hitchin:1990} this was shown using holomorphic Darboux coordinates on the total space of $T_{\cM/S}$, coming from (holomorphic) coordinates on $\cM$.  The reasoning does not strictly speaking need the latter choice though, and it suffices to work with a local trivialisation of $T_{\cM/S}$.  In this sense it also goes through in an algebraic context, as follows.  Let $U_{i}$ be a covering of $\cM$ by open affines, such that $T_{\cM/S}\big|_{U_{i}}$ is free.  For a fixed $i$ we choose generators $e_1, \ldots, e_n$ of the latter.  These can also be understood as functions $f_1, \ldots, f_n$ on $T^*_{\cM/S}\big|_{U_{\gamma}}$.  If we denote the dual sections to $e_1, \ldots, e_n$ as $e^1, \ldots, e^n$, then we can interpret their pull backs as one-forms on the total space of $T^*_{\cM/S}\big|_{U_{\gamma}}$.  The tautological one-form $\theta$ on the total space of $T^*_{\cM/S}$ can now be written locally as $\theta=\sum_{\alpha} f_{\alpha} e^{\alpha}$, and the canonical symplectic form is therefore $\omega=-d\theta=\sum_{\alpha} df_{\alpha}\wedge e^{\alpha}$.  If a section of $\pi_{e*}\Sym^2 T_{\cM/S}$ is locally written as $G=\sum_{\alpha,\beta}G^{\alpha\beta}e_{\alpha}\odot e_{\beta}$ (with the $G^{\alpha\beta}
\in \cO_{U_{i}}$), then the corresponding element of $\pi_{\higgs*}\cO_{\cM^{\higgs}}$ can be written as $\sum_{\alpha,\beta}G^{\alpha\beta}f_{\alpha}f_{\beta}$.  The corresponding Hamiltonian vector field (with respect to $\omega$) in $\pi_{\higgs *}$ is locally written as $$-\sum_{\alpha,\beta,\gamma}e_{\gamma}(G^{\alpha\beta})f_{\alpha} f_{\beta}h^{\gamma}+2\sum_{\alpha,\beta}G^{\alpha\beta}f_{\alpha} e_{\beta},$$  (where, with a slight abuse of notation,  we denote by $e_1, \ldots, e_n, h^1, \ldots, h^n$ the  elements of the basis of $T_{\cM^{\higgs}/S}$ dual to $e^1, \ldots, e^n, df_1, \ldots, df_n$).
After taking the cup product with $\Delta$ (which we represent by a \v{C}ech cohomology class with respect to the open covering $T^*_{U_{i}/S}$), and restricting to $\cO^{(1)}$, this gives indeed $2G\cup \widetilde{\Delta}$.
We conclude by applying this to $\Delta=[\LL]$, in which case the `bottom path' of (\ref{lastdiagram}) is given by a component of the isomorphism $\Psi$. 
 \end{proof}
 \begin{corollary}
 The map $\Phi$ from (\ref{kappaphi}) is an isomorphism.
 \end{corollary}
 \begin{proof}
 This follows immediately by combining Proposition \ref{rho-Hit-isom}, Proposition \ref{cap-isom}, and Proposition \ref{phi-rho-L}.
 \end{proof}
Finally, as a corollary we also get the final fact we need in the proof of the flatness of the Hitchin connection (Theorem \ref{connection-flat}):
\begin{lemma}\label{mu-L-inj}
The map $\mu_{\LL^k}$ is injective.
\end{lemma}

\def\cftil#1{\ifmmode\setbox7\hbox{$\accent"5E#1$}\else
  \setbox7\hbox{\accent"5E#1}\penalty 10000\relax\fi\raise 1\ht7
  \hbox{\lower1.15ex\hbox to 1\wd7{\hss\accent"7E\hss}}\penalty 10000
  \hskip-1\wd7\penalty 10000\box7}
  \def\cftil#1{\ifmmode\setbox7\hbox{$\accent"5E#1$}\else
  \setbox7\hbox{\accent"5E#1}\penalty 10000\relax\fi\raise 1\ht7
  \hbox{\lower1.15ex\hbox to 1\wd7{\hss\accent"7E\hss}}\penalty 10000
  \hskip-1\wd7\penalty 10000\box7}
  \def\cftil#1{\ifmmode\setbox7\hbox{$\accent"5E#1$}\else
  \setbox7\hbox{\accent"5E#1}\penalty 10000\relax\fi\raise 1\ht7
  \hbox{\lower1.15ex\hbox to 1\wd7{\hss\accent"7E\hss}}\penalty 10000
  \hskip-1\wd7\penalty 10000\box7}
  \def\cftil#1{\ifmmode\setbox7\hbox{$\accent"5E#1$}\else
  \setbox7\hbox{\accent"5E#1}\penalty 10000\relax\fi\raise 1\ht7
  \hbox{\lower1.15ex\hbox to 1\wd7{\hss\accent"7E\hss}}\penalty 10000
  \hskip-1\wd7\penalty 10000\box7} \def\cprime{$'$}

\end{document}